\documentclass[a4paper,11pt]{article}
\usepackage{graphicx}
\usepackage{amsmath,amsthm,amssymb,enumerate}%, esint}
\usepackage{euscript,mathrsfs}
\usepackage{dsfont}
\usepackage{xcolor}
\usepackage{empheq}
\usepackage{geometry}

\usepackage{color}
\catcode`\@=11 \@addtoreset{equation}{section}

\catcode`\@=12

\usepackage{stmaryrd}
\allowdisplaybreaks

\usepackage[labelformat=simple]{subcaption}

\usepackage{enumitem}
\usepackage{url}

\usepackage{tikz}
\usetikzlibrary{patterns}
\usepackage[normalem]{ulem}
\usepackage{cancel}

\newtheorem{Theorem}{Theorem}[section]
\newtheorem{Proposition}[Theorem]{Proposition}
\newtheorem{Lemma}[Theorem]{Lemma}
\newtheorem{Corollary}[Theorem]{Corollary}

\theoremstyle{definition}
\newtheorem{Definition}[Theorem]{Definition}

\newtheorem{Remark}[Theorem]{Remark}

\newcommand{\bTheorem}[1]{
\begin{Theorem} \label{T#1} }
\newcommand{\eT}{\end{Theorem}}

\newcommand{\bProposition}[1]{
\begin{Proposition} \label{P#1}}
\newcommand{\eP}{\end{Proposition}}

\newcommand{\bLemma}[1]{
\begin{Lemma} \label{L#1} }
\newcommand{\eL}{\end{Lemma}}
\newcommand{\bCorollary}[1]{
\begin{Corollary} \label{C#1} }
\newcommand{\eC}{\end{Corollary}}

\newcommand{\vrh}{\vr_h}
\newcommand{\vuh}{\vu_h}

\newcommand{\bRemark}[1]{
\begin{Remark} \label{R#1} }
\newcommand{\eR}{\end{Remark}}

\newcommand{\bDefinition}[1]{
\begin{Definition} \label{D#1} }
\newcommand{\eD}{\end{Definition}}

\newcommand{\Nu}{\mathcal{V}}

\newcommand{\vcg}[1]{{\pmb #1}}

\newcommand{\intSh}[1] {\int_{\sigma} #1 \ds }

\newcommand{\bu}{\mathbf u}

\newcommand{\bfu}{\mathbf{u}}
\newcommand{\bfv}{\mathbf{v}}
\newcommand{\bfq}{\mathbf{q}}

\newcommand{\bfe}{\mathbf{e}}

\newcommand{\bfx}{\mathbf{x}}

\newcommand{\bfphi}{\boldsymbol{\phi}}

\newcommand{\Pim}{\Pi_\mathcal{T}}
\newcommand{\Pid}{\Pi_\mathcal{D}}

\newcommand{\sumi}{\sum_{i=1}^d}

\newcommand{\ds}{\,{\rm d}S(x)}

\newcommand{\bFormula}[1]{
\begin{equation} \label{#1}}
\newcommand{\eF}{\end{equation}}

\newcommand{\grid}{\mathcal{T}}
\newcommand{\dgrid}{\mathcal{D}}

\newcommand{\TS}{\Delta t}

\newcommand{\Divh}{{\rm div}_h}

\newcommand{\Gradedge}{\nabla_\faces}
\newcommand{\Gradmesh}{\nabla_{\mathcal{T}}}

\newcommand{\pdedge}{\eth_ \faces}
\newcommand{\pdedgei}{\eth_ \faces^{(i)}}

\newcommand{\pdmesh}{\eth_\grid}
\newcommand{\pdmeshi}{\eth_\grid^{(i)}}

\newcommand{\co}[2]{{\rm co}\{ #1 , #2 \}}
\newcommand{\Ov}[1]{\overline{#1}}

\newcommand{\DC}{C^\infty_c}

\newcommand{\aleq}{\stackrel{<}{\sim}}

\newcommand{\vr}{\varrho}

\newcommand{\vu}{\vc{u}}

\newcommand{\vc}[1]{{\bf #1}}

\newcommand{\Div}{{\rm div}_x}
\newcommand{\Grad}{\nabla_x}

\newcommand{\dx}{\,{\rm d} {x}}

\newcommand{\dt}{\,{\rm d} t }

\newcommand{\jump}[1]{\left\llbracket#1\right\rrbracket}

\newcommand{\norm}[1]{\left\lVert#1\right\rVert}

\newcommand{\intO}[1]{\int_{\Omega} #1 \dx}

\newcommand{\vv}{\vc{v}}
\newcommand{\sumEKh}[1]{ \sum_{K \in \grid} \sum_{\sigma \subset \partial K }\int_{\sigma} #1 \ds }

\newcommand{\I}{\mathbb{I}}

\def\softd{{\leavevmode\setbox1=\hbox{d}%
          \hbox to 1.05\wd1{d\kern-0.4ex{\char039}\hss}}}%cstocs

\definecolor{Cgrey}{rgb}{0.85,0.85,0.85}
\definecolor{Cblue}{rgb}{0.50,0.85,0.85}
\definecolor{Cred}{rgb}{1,0,0}
\definecolor{fancy}{rgb}{0.10,0.85,0.10}
\definecolor{forestgreen}{rgb}{0.13, 0.55, 0.13}

\newcommand\Cbox[2]{%
    \newbox\contentbox%
    \newbox\bkgdbox%
    \setbox\contentbox\hbox to \hsize{%
        \vtop{
            \kern\columnsep
            \hbox to \hsize{%
                \kern\columnsep%
                \advance\hsize by -2\columnsep%
                \setlength{\textwidth}{\hsize}%
                \vbox{
                    \parskip=\baselineskip
                    \parindent=0bp
                    #2
                }%
                \kern\columnsep%
            }%
            \kern\columnsep%
        }%
    }%
    \setbox\bkgdbox\vbox{
        \color{#1}
        \hrule width  \wd\contentbox %
               height \ht\contentbox %
               depth  \dp\contentbox
        \color{black}
    }%
    \wd\bkgdbox=0bp%
    \vbox{\hbox to \hsize{\box\bkgdbox\box\contentbox}}%
    \vskip\baselineskip%
}

\date{}
\newcommand{\pd}{\partial}
\newcommand{\Hc}{\mathcal{H}}
\newcommand{\eps}{\varepsilon}
\newcommand{\faces}{\mathcal{E}}
\newcommand{\facesi}{\faces _i}
\newcommand{\facesK}{\faces(K)}
\newcommand{\facesKi}{\faces _i(K)}
\newcommand{\facesint}{\faces}
\newcommand{\facesinti}{\facesi}

\newcommand{\bQh}{ Q_h}
\newcommand{\bWh}{ W_h}
\newcommand{\myangle}[1]{\langle#1\rangle}

\begin{document}

%%%%%%%%%%%%%%%%%%%%%%%%%%%%%%%%

\title{Convergence of a finite volume scheme \\ for the compressible Navier--Stokes system }

\author{Eduard Feireisl$^{\clubsuit,}$\thanks{E.F.  and H.M. have received funding from
the Czech Sciences Foundation (GA\v CR), Grant Agreement
18--05974S.
B.S. has received funding from the Czech Sciences Foundation (GA\v CR), Grant Agreement 16--03230S.
The Mathematical Institute of the Czech Academy of Sciences is supported by RVO:67985840.
\newline
\hspace*{1em} $^{\spadesuit}$M.L. has been partially supported by the German Science Foundation under the grants TRR~146 Multiscale simulation methods for soft matter systems and TRR~165 Waves to weather.}
\and M\'aria Luk\'a\v{c}ov\'a-Medvi\softd ov\'a $^{\spadesuit}$
\and Hana Mizerov\'a $^{*, \dagger}$
\and Bangwei She $^{*}$
}

\date{}

\maketitle

\centerline{$^*$ Institute of Mathematics of the Academy of Sciences of the Czech Republic}
\centerline{\v Zitn\' a 25, CZ-115 67 Praha 1, Czech Republic}
\centerline{feireisl@math.cas.cz, mizerova@math.cas.cz, she@math.cas.cz}

\bigskip
\centerline{$^\clubsuit$ Technische Universit\"at Berlin}
\centerline{Stra\ss e des 17.~Juni, Berlin, Germany}

\bigskip
\centerline{$^\spadesuit$ Institute of Mathematics, Johannes Gutenberg-University Mainz}
\centerline{Staudingerweg 9, 55 128 Mainz, Germany}
\centerline{lukacova@uni-mainz.de}

\bigskip
\centerline{$^\dagger$ Department of Mathematical Analysis and Numerical Mathematics}
\centerline{Faculty of Mathematics, Physics and Informatics of the Comenius University}

\centerline{Mlynsk\' a dolina, 842 48 Bratislava, Slovakia}

\begin{abstract}
We study convergence of a finite volume scheme for the compressible (barotropic) Navier--Stokes system.
First we prove the energy stability and consistency of the scheme and show that the numerical solutions generate a dissipative measure-valued solution of the system.  Then by the dissipative measure-valued-strong uniqueness principle, we conclude the convergence of the numerical solution to the strong solution as long as the latter exists. Numerical experiments for standard benchmark tests support our theoretical results.
\end{abstract}



%\tableofcontents

\section{Introduction}

We study the flow of a viscous fluid governed by the compressible Navier--Stokes system:
\begin{equation}\label{ns_eqs}
\begin{aligned}
\partial_t \vr + \Div (\vr \vu) &= 0,\\
\partial_t (\vr \vu) + \Div (\vr \vu \otimes \vu) + \Grad p &=  \mu \Delta_x \vu +  (\mu +\lambda) \Grad \Div \vu
\end{aligned}
\end{equation}
in the time--space domain $(0,T)\times\Omega$.
Here $\vr = \vr(t,x),$ and  $\vu = \vu(t,x)$ are the fluid density and velocity,  constants $\mu > 0$, $\lambda\geq -\mu$ are the viscosity coefficients.
The pressure $p$ is assumed to satisfy the \emph{isentropic} state equation
\begin{align}\label{pressure}
p(\vr) = a \vr ^{\gamma} , \quad a>0, \quad \gamma>1.
\end{align}

For the sake of simplicity we impose the periodic boundary conditions, meaning that the domain $\Omega$ can be identified with the flat torus $\Omega = ([0,1]|_{0,1})^d,\ d=1,2,3$.
The problem is (formally) closed by prescribing the initial conditions
\begin{equation}\label{initial_data}
\vr(0) =\vr_0  \in  L^\gamma(\Omega), \ \vr_0>0,\ \vu(0) = \vu_0 \in L^2 (\Omega; R^d).
\end{equation}

In the literature we can find a variety of numerical schemes for viscous compressible flows, such as the Marker--And--Cell scheme~\cite{GGHL, GallouetMAC2018, GallouetMAC, Grapsas, hoseksheMAC},  the finite element schemes~\cite{ABLV, Karper, ZTN2014},  the finite volume schemes~\cite{feist1,  EyGaHe, HJL, MS1998}  and the discontinuous Galerkin schemes~\cite{DoFei, G2018, ID2018}. In this paper we want to concentrate on the finite volume methods that are  standardly used for physical or engineering applications, see,
e.g.~\cite{FK2002, LPK2015, MS1998, PKH2014, WK1996, ZTN2014} and the references therein.  In the cell-centered finite volumes the unknown quantities (numerical solution) are located at the centers of mass of the mesh cells (finite volume cells).  This is very typical for the compressible inviscid flows governed by the Euler equations. By means of the Gauss theorem the inviscid fluxes at cell interfaces are approximated by suitable numerical flux functions.  The latter are based on the flux-vector splitting or upwinding strategy as we will explain below.

For the compressible Navier-Stokes equations in addition the viscous fluxes need to be approximated, which means that the gradients of the numerical solution are to be represented at the cell interfaces. Having piecewise discontinuous approximate functions this requires an additional reconstruction step, which is usually realized by introducing the so-called dual grid around the cell interfaces
of a primary grid. We refer a reader to Kozel et al. \cite{FK2002, LPK2015, PKH2014}, where the viscous terms are approximated by the second order central differences using a dual finite volume grid of octahedrons constructed over
each face of the primary hexagonal finite volume grid. In \cite{MS1998} and \cite{FFL1995} the barycentric subdivision is used to define
dual finite volumes, in \cite{W1997, WK1996} a special reconstruction satisfying maximum principle is developed for the viscous fluxes. A nice overview of various finite volume methods with the gradient approximations at cell interfaces can be found in \cite{B2012}.

Although these methods are frequently used in practical simulations, their convergence for multi-dimen\-sional viscous compressible flows
remains open in general.
  For a mixed finite element--discontinuous Galerkin method, the convergence to a weak solution has been shown by Karper
in his pioneering work \cite{Karper} under the assumption that the adiabatic coefficient $\gamma>3$.
Note that the convergence in this case holds up to a subsequence as the weak solutions are not known to be unique. Moreover,
any generalization of the proof of Karper \cite{Karper} for other numerical schemes, in particular for  cell-centered finite volume methods discussed in the present paper, is highly non-trivial.
In \cite{jovanovic} Jovanovi\'c obtained the error estimate for the isentropic Navier--Stokes equations for entropy dissipative finite volume--finite difference methods under some rather restrictive assumptions on the global smooth solution.
In \cite{FeLu18_CR} Feireisl and Luk\'a\v{c}ov\'a proposed a new way of the convergence proof via  the dissipative measure-valued (DMV) solutions.  They improved the result of \cite{Karper} and showed the convergence of the mixed finite element-finite volume method for the isentropic Navier--Stokes equations for physically relevant range of the adiabatic coefficient $\gamma \in (1,2)$.

We should also mention the recent results  on the analysis of the Marker--And--Cell schemes, cf.~\cite{GallouetMAC2018, GallouetMAC, hoseksheMAC}, which are based on the staggered grid approximation of the velocity and the primary grid approximation of the density.
In \cite{GallouetMAC2018} the convergence to a weak solution of stationary Navier--Stokes equations for $\gamma >3$ has been proved.
In \cite{hoseksheMAC} the consistency and the energy stability of the  Marker--And--Cell scheme has been shown for instationary Navier-Stokes equations. The error estimates for $\gamma > 3/2$ have been presented in \cite{GallouetMAC}  using the relative entropy method.

The main aim of this paper is to demonstrate that the strategy proposed in \cite{FeLu18_CR} can be adapted to investigate  the convergence of  finite volume methods.  More precisely, we consider the first order cell-centered finite volume method, where the
inviscid fluxes are approximated by the upwinding and the viscous fluxes by the central differences. See also our recent works~\cite{FLM18_brenner, FLM18_euler} where analogous finite volume schemes have been applied to show the convergence for the complete Euler system.
We adapt this approach to a time-implicit finite volume method for the barotropic Navier--Stokes system  and show the stability as well as the convergence of numerical solutions to the (unique) strong solution of \eqref{ns_eqs}  provided the latter exists.
The adiabatic coefficient stays in a physically reasonable range $1<\gamma<2$.
To the best of our knowledge, there is no convergence proof of a finite volume method for the multi-dimensional Navier--Stokes system~\eqref{ns_eqs} available in literature assuming only the existence of the  strong solution.

The rest of the paper is organized as follows. In Section \ref{sec_Notations} we introduce the mesh, basic notations, the numerical method, and some preliminary (in)equalities.
Next, in Section~\ref{sec_Stability} we show the energy stability of the scheme and derive all necessary {\it a priori} bounds. Then we establish the consistency  formulation of the scheme in Section~\ref{sec_Consistency}. Further, we address the convergence of approximate solutions in Section~\ref{sec_Convergence}. Finally, we present some numerical experiments in Section~\ref{sec_numerics}.

\section{Numerical scheme}\label{sec_Notations}

We introduce the basic notations, mesh, space and time discretizations,
and, finally, we define the numerical scheme along with some useful (in)equalities.

\subsection{Space discretization}
%\begin{Definition}[Mesh]

{\bf Mesh.}
A discretization of $\Omega$ is given by $\mathcal{M}=(\grid, \faces)$, where:
\begin{itemize}[wide=0pt,topsep=0pt]
\item The primary grid $\grid$ is the set of all compact
 regular quadrilateral elements $K$ such that
\[
\Omega = \bigcup_{K \in \grid} K.
\]
Let $h_i$ be the mesh size in the $i$-th Cartesian direction, and $h=\max_{i=1,\ldots,d} h_i$ be the mesh size.
The mesh is regular in the sense that there exists  a positive $\eta_h$ such that $ \eta_h = \max_{i=1,\ldots,d} \left\{\frac{h}{h_i}\right\}$.

\item We denote by $\faces$ the set of all faces, and by $\facesi$ the set of all faces that  are orthogonal to the standard basis vector $\bfe_i, \ i = 1,\ldots,d,$ of the Cartesian coordinate system.  By  $\facesK$ we  denote the set of faces of an element $K,$ and define $\facesKi=\facesK \cap \facesi$. We further denote by $\vc{n}$ the outer normal vector of a generic face $\sigma \in \faces$. By $\bfx_K$ and  $\bfx_\sigma$ we denote the position of the mass centers of an element $K \in \grid$ and a face $\sigma \in \faces,$ respectively.

\item  The intersection $K\cap L,$ for $K,\ L \in \grid,\ K\neq L$, is either a vertex, or an edge, or a face $\sigma \in \faces$. For any $\sigma \in \faces$ we write $\sigma=K|L$ if  $\sigma=\facesK \cap \faces(L)$, and further write $\sigma =\overrightarrow{K \vert L}$ if
$\bfx_L= \bfx_K +h_i \bfe_i$ or  $\bfx_L= \bfx_K +(h_i- 1) \bfe_i$ for any $\sigma \in \facesi$.
Similarly, we write $K=\overrightarrow{[\sigma \sigma']}$ for $\sigma,  \sigma' \in \facesKi$ if $\bfx_{\sigma'} = \bfx_\sigma + h_i \bfe_i$.
%\dele{or $\bfx_{\sigma'} = \bfx_\sigma + (h_i-1) \bfe_i$. }
For any $\sigma=K|L \in \facesi,\ i\in{1,\ldots,d}$, we also denote  by $d_{\sigma}=h_i$  the periodic distance between the points $\bfx_K$ and $\bfx_L$.

\item  By $|K|$ and $|\sigma|$ we denote the ($d$-- and $(d-1)$--dimensional) Lebesgue measure of an element $K$, and a face $\sigma$, respectively.  Obviously,  $|K| = h_i |\sigma|$ for any $\sigma \in \facesKi$.  In what follows, we shall suppose
\[
|K| \approx  h^d,\ |\sigma| \approx h^{d-1} \ \mbox{for any}\ K \in \grid, \
\sigma \in \faces.
\]
\end{itemize}

\noindent
{\bf Function space.}
In order to define a finite volume scheme we introduce the space of piecewise constant functions $Q_h$ defined on the primary grid $\grid$.
We also introduce a standard projection operator
\begin{equation*}
\Pim: \, L^1(\Omega) \rightarrow Q_h. \quad
\Pim  \phi  = \sum_{K \in \grid} 1_{K} \frac{1}{|K|} \int_K \phi \dx
.
\end{equation*}
For a piecewise (elementwise) continuous function $v$ we define
\[
v^{\rm out}(x) = \lim_{\delta \to 0+} v(x + \delta \vc{n}),\
v^{\rm in}(x) = \lim_{\delta \to 0+} v(x - \delta \vc{n}),\
\Ov{v}(x) = \frac{v^{\rm in}(x) + v^{\rm out}(x) }{2},\
\jump{ v }  = v^{\rm out}(x) - v^{\rm in}(x)
\]
whenever $x \in \sigma \in \facesint$.  Hereafter we mean by $\vv\in Q_h$ that $\vv\in Q_h(\Omega;R^d),$ i.e.,  $v_i\in Q_h,$ for all $i=1,\ldots,d.$

\paragraph{Diffusive upwind flux.}
Given the velocity field $\vv \in \bQh$, the upwind  flux for any function $r\in Q_h$ is defined at each face  $\sigma \in \facesint$ by
\begin{align*}\label{Up}
Up [r, \vv]   =r^{\rm up} \vv \cdot \vc{n}
=r^{\rm in} [\Ov{\vv} \cdot \vc{n}]^+ + r^{\rm out} [\Ov{\vv} \cdot \vc{n}]^-
= \Ov{r} \ \Ov{\vv} \cdot \vc{n} - \frac{1}{2} |\Ov{\vv} \cdot \vc{n}| \jump{r},
\end{align*}
where
\begin{equation*}
[f]^{\pm} = \frac{f \pm |f| }{2} \quad \mbox{and} \quad
r^{\rm up} =
\begin{cases}
 r^{\rm in} & \mbox{if} \ \Ov{\vv} \cdot \vc{n} \geq 0, \\
r^{\rm out} & \mbox{if} \ \Ov{\vv} \cdot \vc{n} < 0.
\end{cases}
\end{equation*}
Furthermore, we consider a diffusive numerical flux function of the following form
\begin{align}\label{num_flux}
F_h(r,\vv)
={Up}[r, \vv] - h^{\eps} \jump{ r }, \, \eps>0.
\end{align}
When $r$ is a vector function, e.g. $r= \vr \vu$ in the momentum equation, we write the above numerical flux in bold font as $ {\bf F}_h(\vr \vu, \vv) \equiv \big( F_h(\vr u_1, \vv), \ldots, F_h (\vr u_d, \vv)  \big)^T$  and $ {\bf Up}(\vr \vu, \vv) \equiv \big( Up(\vr u_1, \vv), \ldots, Up (\vr u_d, \vv)  \big)^T$.

{\bf Discrete divergence.}
We define the discrete divergence operator as
\begin{equation}\label{div_mesh}
\Divh \vuh (\bfx) := \sum_{K \in \grid}  (\Divh \vuh)_K 1_K, \quad (\Divh \vuh)_K :=
\frac{1}{|K|}\sum_{\sigma\in \facesK}|\sigma| \Ov{\vuh} \cdot \vc{n},\  \mbox{ for all } \ \vuh \in Q_h.
\end{equation}
%\begin{Remark}
%We would like to point out that it is easy to make a mistake by defining the discrete divergence as
%\[ (\widetilde{\Divh} \bfv )_\sigma =  \frac{\jump{\bfv}}{d_\sigma} \cdot \vc{n}.
%\]
%Our definition of the discrete divergence operator $\Divh$ defined in \eqref{div_mesh} can be  interpreted as the averaging of the above.
%\[ (\Divh \bfv )_K = \frac{1}{|K|} \sum_{\sigma \in \facesK} |\sigma| \Ov{\bfv} \cdot \vc{n}
%=  \frac{1}{|K|} \sum_{\sigma \in \facesK} |\sigma| \jump{\bfv} \cdot \vc{n}
%=   \sum_{\sigma \in \facesK}  \frac{|D_\sigma|}{|K|} (\widetilde{\Divh} \bfv )_\sigma ,
%\]
%where we have used the identity $\sum_{\sigma \in \facesK} |\sigma|  \vc{n}  =  \int_K \Div \vc{1} =0. $
%
%\end{Remark}

\subsection{Time discretization}
For a given time step $\TS \approx h>0$,
we denote the approximation of a function $v_h$ at time $t^k= k\TS$ by $v_h^k$ for $k=1,\ldots,N_T(=T/\TS)$.  The time derivative is discretized by the backward Euler method,
\[
  D_t v_h^k = \frac{v_h^k- v_h^{k-1}}{\TS},\ \mbox{ for } k=1,2,\ldots, N_T.
\]
Furthermore,  we introduce the piecewise constant extension of discrete values,
\begin{equation*}
\begin{aligned}
\vrh(t,\cdot) &=\vrh^0 \mbox{ for } t<\TS,\ &\vrh(t,\cdot)=\vrh^k \mbox{ for } t\in [k\TS,(k+1)\TS),\ k=1,2,\ldots,N_T,\\
\vuh(t,\cdot) &=\vuh^0 \mbox{ for } t<\TS,\ &\vuh(t,\cdot)=\vuh^k \mbox{ for } t\in [k\TS,(k+1)\TS),\ k=1,2,\ldots,N_T,
%\\ \mh(t,\cdot) &=\mh^0 \mbox{ for } t<\TS,\ \mh(t,\cdot)&=\vrh^k \vuh^k \mbox{ for } t\in [k\TS,(k+1)\TS),\ k=1,2,\ldots,N_T,
\end{aligned}
\end{equation*}
and $p_h=p(\vrh),$  for which the discrete time derivative then reads
\[
 D_t v_h = \frac{v_h (t,\cdot) - v_h(t - \Delta t,\cdot)}{\TS} .
\]
We shall write  $A \aleq B$ if $A \leq cB$ for a generic positive constant $c$ independent of $h.$

\subsection{Numerical scheme}
Using the above notation we introduce the  implicit  finite volume scheme to approximate system \eqref{ns_eqs}.
\begin{Definition}[Numerical scheme]
  Given the initial values  $(\vrh^0,\vuh^0) =(\Pim\vr_0, \Pim\vu_0),$ find $(\vr_h,\vu_h) \in Q_h\times Q_h$ satisfying for $k=1,\ldots,N_T$ the following equations
\begin{subequations}\label{scheme}
\begin{align}
&\intO{ D_t \vrh^k \phi_h } - \sum_{ \sigma \in \facesint } \intSh{  F_h(\vrh^k,\vuh^k)
\jump{\phi_h}   } = 0 \quad \mbox{for all } \phi_h \in Q_h,\label{scheme_den}\\
&\intO{ D_t  (\vrh^k \vuh^k) \cdot \bfphi_h } - \sum_{ \sigma \in \facesint } \intSh{ {\bf F}_h(\vrh^k \vuh^k,\vuh^k)
\cdot \jump{\bfphi_h}   }- \sum_{ \sigma \in \facesint } \intSh{ \Ov{p_h^k} \vc{n} \cdot \jump{ \bfphi_h }  } \nonumber \\
&= - \mu \sum_{ \sigma \in \facesint } \intSh{ \frac{1}{d_\sigma}\jump{ \vuh^k }  \cdot \jump{ \bfphi_h }  }
- (\mu +\lambda)  \intO{\Divh   \vuh^k  \; \Divh \bfphi_h }
\quad \mbox{for all }
\bfphi_h \in \bQh. \label{scheme_mom}
\end{align}
\end{subequations}
\end{Definition}
 The weak formulation \eqref{scheme}  of the scheme can be rewritten in the standard  per cell finite volume formulation
for all $K \in \grid,$
\begin{equation}\label{scheme_fv}
\begin{aligned}
&D_t \vr^k _K + \sum_{\sigma \in \facesK} \frac{|\sigma|}{|K|} F_h(\vrh^k,\vuh^k) =0,
 \\
&D_t (\vrh^k \vuh^k)_K + \sum_{\sigma \in \facesK} \frac{|\sigma|}{|K|}
\left({\bf F}_h(\vrh^k \vuh^k,\vuh^k)  + \Ov{p_h^k} \vc{n}
- \mu \frac{\jump{\vuh^k}}{d_\sigma}
-(\mu + \lambda) \Ov{\Divh \vuh^k} \vc{n}\right) =0.
\end{aligned}
\end{equation}

\begin{Remark}
Let us explain the role of $h^\eps$-terms in  \eqref{scheme_fv}. Clearly, they are additional diffusion terms
\[ \sum_{\sigma \in \facesK} \frac{|\sigma|}{|K|} h^\eps\jump{r_h} =h^{\eps+1} (\Delta_h r_h)_K.
\]
In this paper we require $ 0 < \eps < \min\{1, 2(\gamma -1)\}   $  as a compromise
between minimality of the additional numerical diffusion and necessary consistency estimates, see Section~4.
\end{Remark}

The approximate solutions resulting from scheme \eqref{scheme} enjoy the following properties:

\begin{itemize}[wide=0pt]
\item[1.] {\bf Conservation of mass.}\\
Taking $\phi_h \equiv 1$ in the equation of continuity (\ref{scheme_den}) yields the total mass conservation
\begin{equation*}
\intO{ \vrh (t, \cdot) } = \intO{ \vr_{h}^0 } > 0,\ t \geq 0.
\end{equation*}
\item[2.] {\bf Existence of numerical solution.}\\
The discrete problem
\eqref{scheme} admits a solution $(\vrh^k, \vuh^k)$  for any $k = 1,\ldots,N_T$. We refer a reader to \cite[Theorem 3.5]{hoseksheMAC} for the proof, as it can be done exactly in the same way.
\item[3.] {\bf Positivity of numerical density.} \\
Any solution $(\vrh^k, \vuh^k)$ to
\eqref{scheme} satisfies $\vrh^k>0$ provided $\vrh^{k-1}>0$,  $k = 1,\ldots, N_T$, see \cite[Lemma 3.2]{hoseksheMAC} for the proof.
%\item[4.] For any $K\in \grid$, the scheme can be written in finite volume setting as
%\begin{equation*} %\label{scheme_fv}
%\begin{split}
% D_t \vr_K^k  \textcolor{red}{+} \sum_{ \sigma \in \facesK } \frac{|\sigma|}{|K|}  F(\vrh^k,\vuh^k)  = 0 ,
%\\
% D_t \vm_K^k  + \sum_{ \sigma \in \facesK } \frac{|\sigma|}{|K|} \left( {\bf F}_h(\vrh^k \vuh^k,\vuh^k)
% +  \Ov{p_h^k} \vc{n} \textcolor{red}{-} \mu  \frac{\jump{ \vuh^k } }{d_\sigma}
%\textcolor{red}{-} (\mu +\lambda)  \frac{\jump{ \vuh^k }  \cdot \vn}{d_\sigma}   \vn \right) =0.
%\end{split}
%\end{equation*}
\end{itemize}

\subsection{Preliminaries}

To investigate theoretical properties of our finite volume method it is convenient to define a dual grid.
We emphasize that the dual grid is not needed for the implementation of the scheme.

\noindent {\bf Dual grid.} A dual element  $D_\sigma$ is associated to a generic face $\sigma=K|L\in \facesint$, where $D_\sigma = D_{\sigma,K} \cup D_{\sigma,L}$, and $D_{\sigma,K}$ (resp.~$D_{\sigma,L}$) is built by half of $K$ (resp.~$L$), see Figure~\ref{fig:mesh} for an example of such cell.  We denote the set of all dual cells as $\dgrid$.
Furthermore, we define $\dgrid_i = \{ D_\sigma \}_{\sigma \in \facesi}, i = 1,\ldots, d$. Note that for each $i$ it holds that
$ \Omega = \bigcup_{\sigma \in \facesi} D_\sigma $.

\begin{figure}[!h]
\centering
\begin{tikzpicture}[scale=1.0]
\draw[-,very thick](0,-2)--(5,-2)--(5,2)--(0,2)--(0,-2)--(-5,-2)--(-5,2)--(0,2);

%\draw[-,very thick](0,-2)--(0,2);
%\draw[fill=blue!10](0,-2)--(0,2)--(-2.5,2)--(-2.5,-2)--(0,-2);
%\draw[fill=green!30, pattern=northeast](0,-2)--(2.5,-2)--(2.5,2)--(0,2)--(0,-2);
\draw[-,very thick, green=90!, pattern=north west lines, pattern color=green!30] (0,-2)--(2.5,-2)--(2.5,2)--(0,2)--(0,-2);
\draw[-,very thick, blue=90!, pattern= north east  lines, pattern color=blue!30] (0,-2)--(0,2)--(-2.5,2)--(-2.5,-2)--(0,-2);

\path node at (-3.5,0) { $K$};
\path node at (3.5,0)  { $L$};
\path node at (-2.5,0) {$ \bullet$};
\path node at (-2.8,-0.3) {$ \bfx_K$};
\path node at (2.5,0) {$\bullet$};
\path node at (2.7,-0.3) {$ \bfx_L$};
\path node at (0,0) {$\bullet$};
\path node at (0.3,-0.3) {$ \bfx_\sigma$};

\path (-0.4,0.8) node[rotate=90] { $\sigma=\overrightarrow{K|L}$};
 \path (-1.5,1.4) node[] { $D_{\sigma,K}$};
 \path (1.5,1.4) node[] { $D_{\sigma,L}$};
 \end{tikzpicture}
\caption{Dual grid}
 \label{fig:mesh}
\end{figure}
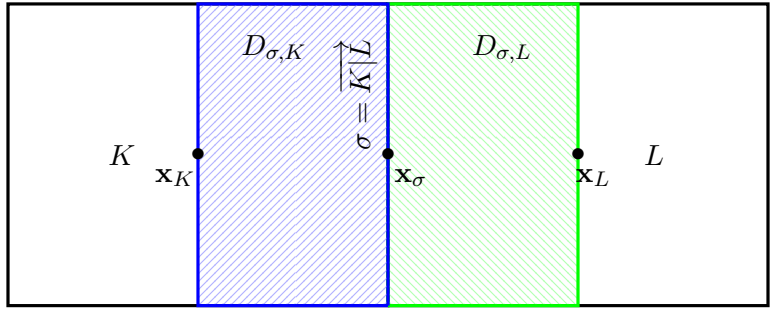
Let $W_h^{(i)},$ $i = 1,\ldots,d,$ be the space of piecewise constant functions defined on the dual grid $\dgrid _i$. By $\bfq=(q_{1}, \ldots, q_{d}) \in \bWh:=\big( W_h^{(1)}, \ldots,  W_h^{(d)} \big)$ we mean that $q_{i} \in W_h^{(i)}$, for all $i=1, \ldots, d$.
We define the standard projection of $\phi \in L^1(\Omega)$  into the discrete functional spaces  $W_h,$
%%%%%%%%%%%%%%%%%%%%% corrected ML
%%\begin{equation*}
%%\Pid: \quad W^{1,1}(\Omega) \rightarrow W_h. \quad
%%\Pid  = (\Pid^{(1)}, \ldots, \Pid^{(d)}), \quad \Pid^{(i)} \phi
%%= \sum_{\sigma \in \facesi}  \frac{1_{D_\sigma}}{|\sigma|} \int_\sigma \phi \ds .
%%\end{equation*}
%%%%%%%%%%%%%%%%%%%%%%%%%%%%%%%
\begin{equation*}
\Pid: \quad L^1(\Omega) \rightarrow W_h, \quad
\Pid  = (\Pid^{(1)}, \ldots, \Pid^{(d)}), \quad \Pid^{(i)} \phi
= \sum_{\sigma \in \facesi}  \frac{1_{D_\sigma}}{|D_\sigma|} \int_{D_\sigma} \phi \dx .
\end{equation*}

\noindent{\bf Discrete differential operators.} We need some discrete operators that are not directly used to discretize the Navier-Stokes system, but are essential to establish the consistency formulation in Section~\ref{sec_Consistency}. For any $r_h \in Q_h$ and $\bfq_h=(q_{1,h},\ldots, q_{d,h})\in W_h$, we define the  difference operators based on the dual grid
\begin{equation*}
\pdedgei r_h(\bfx) := \sum_{\sigma \in \facesi} 1_{D_\sigma}\left(\pdedge^{(i)}r_h\right)_{D_\sigma},
\quad  \left(\pdedge^{(i)}r_h\right)_{D_\sigma} :=\frac{r_L - r_K}{ d_\sigma}, \; \mbox{ for all } \; \sigma =\overrightarrow{K|L} \in \facesinti,
\end{equation*}
and the primary grid
\begin{equation*}
\pdmeshi q_{ i,h}(\bfx) := \sum_{K \in \grid} \left(\pdmesh^{(i)}q_{i,h}\right)_{K}1_K,\; \ i = 1,\ldots, d,
\end{equation*}
where
\[\left(\pdmesh^{(i)}q_{i,h} \right)_{K} :=\frac{ q_{i,h}|_{\sigma'} - q_{i,h}|_{\sigma}}{h} , \; \mbox{ for all } \; \sigma, \sigma' \in \facesinti \mbox{ and } K=\overrightarrow{[\sigma \sigma']}.
\]
Using the above notations we define the gradient operators for $\ r_h \in Q_h$ and $\bfq_h \in \bWh$  by
\begin{equation*}
\Gradedge r_h(\bfx) := (\pdedge ^{(1)}r_h, \ldots, \pdedge ^{(d)} r_h )(\bfx)
\quad \mbox{ and } \quad
\Gradmesh \bfq_{h}  := (\pdmesh ^{(1)} q_{1,h},  \ldots, \pdmesh^{(d)} q_{d,h})(\bfx),
\end{equation*}
respectively.
Note that the divergence operator $\Divh$ defined in \eqref{div_mesh} can be  rewritten  for all $\vuh  \in Q_h $
\begin{equation} \label{divdiv}
\Divh  \vuh = \sum_{i=1}^d \pdmeshi \Ov{u_{i,h}},
\end{equation}
which for  a regular rectangular grid is equivalent to
\begin{equation*}
\Divh  \vuh = \sum_{i=1}^d \pdmeshi \left( \Pid^{(i)} \vuh \right)  .
\end{equation*}
Moreover, we define the discrete Laplace operator for $r_h \in Q_h$ on the primary grid
\begin{equation*}
 \Delta_h r_h(\bfx) = \sumi \Delta_h^{(i)} r_h(\bfx)=  \sum_{K\in \grid} (\Delta_h r_h)_K 1_K , \quad
 \Delta_h^{(i)} r_h(\bfx) = \sum_{K\in \grid} (\Delta_h^{(i)} r_h)_K 1_K ,
\end{equation*}
where $i =  1, \ldots, d$, and
\[  (\Delta_h^{(i)} r_h)_K :=  \frac{1}{|K|} \sum_{\sigma \in \facesKi}  |\sigma| \frac{\jump{r_h}}{d_\sigma},
  \quad (\Delta_h r_h)_K :=
 \frac{1}{|K|} \sum_{\sigma \in \facesK}  |\sigma| \frac{\jump{r_h}}{d_\sigma}, \; \mbox{ for all } \; K \in \grid.
\]
In addition, it is worth mentioning that
\[  \Delta_h^{(i)} r_h  = \pdmeshi ( \pdedgei r_h) ,  \, i = 1, \ldots, d.
\]

\noindent{\bf Integration by parts.}
 Let us start with recalling the algebraic identity
\begin{equation}\label{avg_diff}
\Ov{u_h v_h} - \Ov{u_h}\ \Ov{v_h} =\frac14 \jump{u_h} \jump{v_h}
\end{equation}
together with the product rule
\begin{equation}\label{product_rule}
\jump{ u_h v_h }  = \Ov{u_h} \jump{v_h}  + \jump{u_h}  \Ov{v_h}\; ,
\end{equation}
which are valid for any $u_h, v_h\, \in Q_h.$
A direct application of the product rule \eqref{product_rule} further implies
 \begin{equation}\label{basic_eq1}
\jump{r_h \bfv_h} \cdot \jump{\bfv_h} - \frac12 \jump{r_h} \jump{|\bfv_h|^2} =\Ov{r_h} |\jump{\bfv_h}|^2
 \mbox{ for } r_h \in Q_h,\ \bfv_h \in \bQh,
\end{equation}
and  the following lemma.
\begin{Lemma}  For  any $ r_h\in Q_h$ and $ \bfv_h \in \bQh$ it holds
\begin{equation}\label{basic_eq2}
 \sum_{ \sigma \in \facesint } \intSh{ \left( \Ov{r_h} \jump{ \bfv_h } +\Ov{\bfv_h }  \jump{r_h  }   \right) \cdot \vc{n} }
=0.
\end{equation}
\end{Lemma}
\begin{proof}  For the functions $r_h,$ $\vv_h$ that are constant on each element $K\in\grid$,  it holds that
\[ \sum_{ \sigma \in \facesint }  \intSh{ \left( \Ov{r_h} \jump{ \bfv_h } +\Ov{\bfv_h }  \jump{r_h  }   \right) \cdot \vc{n} }
=  \sum_{ \sigma \in \facesint } \intSh{  \jump{r_h  \bfv_h}   \cdot \vc{n} }
=-\sum_{K\in \grid}  r_K \bfv_K \cdot \sum_{ \sigma \in \facesK } \int_\sigma  \vc{n} \ds
=0.
\]
\end{proof}
Consequently,  for any $r_h, \phi_h \in Q_h$ and $ \bfq_h \in W_h$, it is easy to
 observe the following  discrete integration by parts formulae
\begin{subequations}\label{int_by_part}
\begin{equation}\label{int_by_part_las}
\intO{\Delta_h r_h \phi_h}
=-\intO{ \Gradedge r_h \cdot \Gradedge \phi_h }= \intO{ r_h \Delta_h \phi_h},
\end{equation}
\begin{equation}\label{int_by_part_gradm}
\intO{q_{i,h} \pdedgei r_h} = - \intO{  r_h \pdmeshi q_{i,h}},\;\mbox{ for all }\; i = 1,\ldots,d.
\end{equation}
%We note that \eqref{int_by_part_gradm} is just another form of \eqref{basic_eq2}.
\end{subequations}
 {\bf Useful estimates.}
Next, we list some basic inequalities used  in the numerical analysis.
We assume the reader is fairly familiar with this matter, for which we refer to the monograph \cite{EyGaHe}, and the article paper \cite{GallouetMAC}. If
$\phi \in C^1(\Ov{\Omega})$ we have
\begin{equation} \label{n4c}
  \Big|  \jump{ \Pim  \phi  }   \Big|_{\sigma} \lesssim h \| \phi \|_{C^1},\
\  \mbox{for any}\ x \in \sigma \in \facesint, \mbox{ and }
\norm{\phi - \Pim  \phi }_{L^p(\Omega)} \lesssim h \| \phi \|_{C^1}.
\end{equation}
Furthermore, if $\phi \in C^2(\Ov{\Omega})$  we have for all $1< p \leq \infty$
\begin{align}
\norm{ \Grad \phi - \Gradedge \Pim \phi  }_{L^p}  \lesssim h, \quad
\norm{ \Gradedge \Pim \phi  }_{L^p}  \lesssim  \norm{\phi}_{C^1} +h,
\label{n4c2_edge}\\
\norm{ \Grad \phi - \Gradmesh  \Pid \big(\Pim \phi\big)  }_{L^p}  \lesssim h, \quad
\norm{ \Div \phi - \Divh  (\Pim \phi )   }_{L^p}  \lesssim h.\label{n4c2_mesh}
\end{align}
 If in addition,  $\phi \in C^3(\Ov{\Omega})$ we get
\begin{equation} \label{n4c3}
 \norm{ \Delta_h \Pim  \phi  - \Delta_x \phi }_{L^p} \lesssim h \| \phi \|_{C^3}, \quad
 \norm{ \Delta_h \Pim  \phi  }_{L^p} \lesssim  \| \phi \|_{C^2} + h\| \phi \|_{C^3} , \, \mbox{ for all } \ 1< p \leq \infty.
\end{equation}
 The inverse estimates \cite{ciarlet} for $r_h \in Q_h $ read
\begin{equation} \label{inverse_inequality}
\| r_h \|_{L^p(\Omega)} \lesssim h^{ d (\frac1p -\frac1q)  } \|r_h \|_{L^q(\Omega)} \ \mbox{ for any }\  1 \leq q \leq p \leq \infty.
\end{equation}
 Finally, we need a discrete analogous of the Sobolev-type inequality that can be proved exactly as \cite[Theorem 11.23]{FeiNov_book}.
\begin{Lemma}[Sobolev inequality]
Let the function $r \geq 0$ be such that
\[ 0< \intO{r} =c_M, \mbox{ and } \intO{r^\gamma}\leq  c_E \mbox{ for } \gamma>1,
\]
where $c_M$ and $c_E$ are some positive constants.
Then the following Poincar{\'e}-Sobolev type inequality holds true
\begin{equation}\label{sobolev_inequ}
\norm{ v_h }_{L^6(\Omega)} \leq c \norm{\Gradedge v_h }_{L^2(\Omega)}^2
+c \left(\intO{r|v_h |}\right)^2
\aleq  c \norm{\Gradedge v_h }_{L^2(\Omega)}^2  +c_M
+c \intO{r|v_h |^2}
\end{equation}
for any $v_h \in Q_h$, where the constant $c$ depends on $c_M$ and $c_E$ but not on the mesh parameter.
\end{Lemma}

The following lemma shall be useful for analysing the error between the continuous convective term and its numerical analogue.
\begin{Lemma}\label{lem_convective_trans}
For any $r_h , \bfv_h  \in \bQh$, and $\phi \in C^1(\Omega)$, it holds
\begin{align*}
&\intO{ r_h \bfv_h  \cdot \Grad \phi } - \sum_{\sigma \in \facesint} \intSh{ F_h [ r_h,\bfv_h  ] \ \jump{\Pim \phi }  }
\\&=
 \sum_{\sigma \in \facesint} \intSh{
\left( \frac{1}{2} |\Ov{\bfv_h } \cdot \vc{n}|  + h^\eps  + \frac{1}{4} \jump{\bfv_h } \cdot \vc{n}   \right) \jump{r_h }  \jump{\Pim  \phi }  }
+ \intO{ r_h \bfv_h  \cdot \left(\Grad \phi - \Gradmesh \Pid \big(\Pim  \phi\big) \right)}.
\end{align*}
\end{Lemma}
\begin{proof}Using the basic equalities \eqref{avg_diff}--\eqref{basic_eq2}, we have
\begin{align*}
\intO{ r_h \bfv_h  \cdot \Grad \phi } &= \sum_{K \in \grid} \int_{K} r_h \bfv_h  \cdot \Grad \phi \ \dx
\\&
= \sum_{K \in \grid} \int_{K} r_h \bfv_h  \cdot (\Grad \phi - \Gradmesh \Pid \big(\Pim  \phi\big)) \dx
+ \sum_{K \in \grid} \int_{\partial K} r_h \bfv_h  \cdot \vc{n} \Ov{\Pim  \phi }  \ds
\\&=
 \intO{ r_h \bfv_h  \cdot (\Grad \phi - \Gradmesh \Pid \big(\Pim  \phi\big) )}
 - \sum_{\sigma \in \facesint} \intSh{ \jump{ r_h \bfv_h  }  \cdot \vc{n} \Ov{\Pim  \phi  } }
\\&=
 \intO{ r_h \bfv_h  \cdot (\Grad \phi - \Gradmesh \Pid \big(\Pim  \phi\big) )}
+ \sum_{\sigma \in \facesint} \intSh{ \Ov{ r_h \bfv_h  }  \cdot \vc{n} \jump{ \Pim  \phi  }  }
\\&=
 \intO{ r_h \bfv_h  \cdot (\Grad \phi - \Gradmesh \Pid \big(\Pim  \phi\big) )}
+ \sum_{\sigma \in \facesint} \intSh{ \left( \Ov{ r_h \bfv_h }   - \Ov{ r_h } \ \Ov{\bfv_h } \right) \cdot \vc{n} \jump{ \Pim  \phi  }  }
\\& \qquad
+\sum_{\sigma \in \facesint} \intSh{ \Ov{ r_h } \ \Ov{\bfv_h }  \cdot \vc{n} \jump{\Pim  \phi }  }
\pm
 \sum_{\sigma \in \facesint} \intSh{
\left( \frac{1}{2} |\Ov{\bfv_h } \cdot \vc{n}|  + h^\eps   \right) \jump{r_h }  \jump{ \Pim  \phi }  }
\\&=
 \intO{ r_h \bfv_h  \cdot (\Grad \phi - \Gradmesh\Pid \big(\Pim  \phi\big) )}
+\sum_{\sigma \in \facesint} \intSh{ \frac14 \jump{ r_h } \jump{\bfv_h }  \cdot \vc{n} \jump{ \Pim  \phi  }  }
\\& \qquad
+
\sum_{\sigma \in \facesint} \intSh{ F_h [ r_h,\bfv_h  ]  \jump{\Pim \phi }  } +
\sum_{\sigma \in \facesint} \intSh{\left( \frac{1}{2} |\Ov{\bfv_h } \cdot \vc{n}|  + h^\eps   \right) \jump{r_h }  \jump{ \Pim  \phi   }} .
\end{align*}
\end{proof}

\section{Stability}\label{sec_Stability}

 In this section we show the stability of the scheme and derive the energy estimates that will be necessary for the consistency formulation in Section~\ref{sec_Consistency}. For simplicity,   we will hereafter denote the norms $\norm{\cdot}_{L^q(\Omega)}$ and $\norm{\cdot}_{L^p(0,T;L^q(\Omega))}$ by $\norm{\cdot}_{L^q}$ and $\norm{\cdot}_{L^pL^q}$, respectively.  We also denote $\co{A}{B}=[\min\{A,B\}, \max\{A,B\}]$.

%To begin, we introduce two renormalized equations.  that $\Hc(\vr)=\frac{p(\vr)}{\gamma-1}$
 To begin, we recall the  discrete internal energy balance, which is a result of the renormalization of the continuity equation, see, e.g.~\cite[Section 4.1]{FeKaNo} or \cite[Lemma 3.1]{hoseksheMAC}. Indeed, multiplying \eqref{scheme_den} by $\Hc'(\vrh^k)$ of the internal energy $\Hc(\vr)=\frac{p(\vr)}{\gamma-1}$ gives rise to the result of the following lemma.
\begin{Lemma}[ Discrete internal energy balance]\label{lem_r1}
Let $(\vrh,\vuh)\in  Q_h \times Q_h$ satisfy the discrete continuity equation \eqref{scheme_den}. Then  there exists $\xi \in \co{\vrh^{k-1}}{\vrh^k}$ and  $\zeta \in \co{\vr_K^k}{\vr_L^k}$ for any $\sigma=K|L \in \facesint$  such that
\begin{equation}\label{r1}
\begin{split}
\intO{ D_t \Hc(\vrh^k)  }
&- \sum_{ \sigma \in \facesint } \intSh{ \Ov{\vuh^k} \cdot \vc{n}  \jump{p(\vrh^k)}  }
\\ &=
- \frac{\TS}2 \intO{ \Hc''(\xi)|D_t \vrh^k|^2  }
- \frac12 \sum_{ \sigma \in \facesint } \intSh{ \Hc''(\zeta) \jump{  \vrh^k } ^2 \left(h^\eps  + | \Ov{\vuh} \cdot \vc{n} | \right)
}
.
\end{split}
\end{equation}
%where $\Hc (\vr) =\frac{p(\vr)}{\gamma-1}$.
\end{Lemma}
Next, we recall the renormalization of the transport equation, see \cite[Lemma A.1, Section A.2]{FeKaNo}.
\begin{Lemma}[Discrete renormalized transport equation]
Suppose that  $b_h^k \in Q_h,\ \chi \in C^2(R)$. Then there exists $\xi \in \co{b_h^{k-1}}{ b_h^k},\  \zeta\in \co{b_h^{k}}{(b_h^k)^{\rm out}}$ for any $\phi_h \in Q_h$ such that
\begin{equation} \label{r3}
\begin{split}
&\intO{ D_t (\vrh^k b_h^k) \chi'(b_h^k) \phi_h } - \sum_{ \sigma \in \facesint } \intSh{ {Up}[\vrh^k b_h^k, \vuh^k] \jump{\chi' (b_h^k) \phi_h}  }\\
&= \intO{ D_t \left(\vrh^k \chi(b_h^k) \right) \phi_h } - \sum_{ \sigma \in \facesint } \intSh{ {Up}[\vrh^k \chi (b_h^k), \vuh] \jump{ \phi_h}  }
+\frac{\TS}{2} \intO{ \chi''(\xi)\vrh^{k-1}|D_t b_h^k|^2  \phi_h}
\\
&\quad +\sum_{ \sigma \in \facesint } \intSh{ h^\eps  \jump{\vrh^k} \ \jump{ \left( \chi(b_h^k) - \chi'(b_h^k) b_h^k \right) \phi_h }  }
\\&\quad
-\frac12 \sumEKh{\phi_h  \chi''(\zeta) \jump{b_h^k}^2 (\vrh^k)^{\rm out}  \left[\Ov{\vuh^k} \cdot \vc{n}\right]^- }.
\end{split}
\end{equation}
\end{Lemma}

 \subsection{Total energy balance}
Now, we are ready to  derive the discrete counterpart of the total energy balance.
\begin{Theorem}[Discrete energy  balance]\label{theorem_stability}
Let $(\vrh,\vuh)$ be a numerical solution obtained from scheme \eqref{scheme}. Then, for any $k=1,\ldots,N_T$,  there exists $\xi \in \co{\vrh^{k-1}}{\vrh^k}$ and  $\zeta \in \co{\vr_K^k}{\vr_L^k}$ such that,  for any $\sigma=K|L \in \facesint,$
\begin{equation} \label{energy_stability}
\begin{split}
& D_t \intO{ \left(\frac{1}{2} \vrh^k |\vuh^k|^2  + \Hc(\vrh^k) \right)  }
+ h^\eps  \sum_{ \sigma \in \facesint } \intSh{  \Ov{\vrh^k} \jump{ \vuh^k}^2  }
+  \mu   \norm{\Gradedge \vuh^k}_{L^2}^2  + (\mu +\lambda) \intO{  |\Divh \vuh^k|^2}
\\& =
- \frac{\TS}2 \intO{ \Hc''(\xi)|D_t \vrh^k|^2  }
- \frac12 \sum_{ \sigma \in \facesint } \intSh{ \Hc''(\zeta) \jump{  \vrh^k } ^2 \left(h^\eps  + | \Ov{\vuh^k} \cdot \vc{n} | \right)
}
\\&
\quad - \frac{\TS}{2} \intO{ \vrh^{k-1}|D_t \vuh^k|^2  }
- \frac12 \sum_{ \sigma \in \facesint } \intSh{ (\vrh^k)^{\rm up} |\Ov{\vuh^k} \cdot \vc{n} |\jump{ \vuh^k } ^2   }
.
\end{split}
\end{equation}
\end{Theorem}
\begin{proof}
First, taking $\bfphi_h = \vuh^k $ in \eqref{scheme_mom} we get
\[
\begin{split}
& \intO{D_t (\vrh^k \vuh^k) \cdot \vuh^k  } - \sum_{ \sigma \in \facesint } \intSh{ {\bf F}_h(\vrh^k \vuh^k,\vuh^k)
\cdot \jump{\vuh^k  }   } - \sum_{ \sigma \in \facesint } \intSh{ \Ov{p_h^k} \vc{n} \cdot \jump{ \vuh^k }  } \\
&=
-  \mu   \norm{\Gradedge \vuh^k}_{L^2}  - (\mu +\lambda) \intO{  |\Divh \vuh^k|^2} .
\end{split}
\]
Next, we use relation (\ref{r3}) for $b_h = \vuh^k$,  $\chi(|\vuh^k|) = \frac{1}{2} |\vuh^k|^2$, and $\phi_h=1$ to compute
\[
\begin{split}
\\&
\intO{ D_t (\vrh^k \vuh^k) \cdot \vuh^k  } - \sum_{ \sigma \in \facesint } \intSh{  {\bf Up}[\vrh^k \vuh^k, \vuh^k ] \cdot \jump{ \vuh^k }  }
\\&=
\intO{D_t \left( \frac{1}{2} \vrh^k |\vuh^k|^2 \right)} - \sum_{ \sigma \in \facesint } \intSh{ {\bf Up} \left[ \frac{1}{2} \vrh^k |\vuh^k|^2 , \vuh^k \right] \underbrace{\jump{ 1} }_{=0} }
 +\frac{\TS}{2} \intO{ \vrh^{k-1}|D_t \vuh^k|^2  }
\\&\quad
-\sum_{ \sigma \in \facesint } \intSh{ h^\eps  \jump{\vrh^k} \jump{\frac{1}{2}|\vuh^k|^2}    }
- \frac{1}{2} \sumEKh{
 (\vrh^k)^{\rm out} \left[\Ov{\vuh^k} \cdot \vc{n} \right]^-\jump{ \vuh^k }^2   }
\\&=
\intO{D_t \left( \frac{1}{2} \vrh^k |\vuh^k|^2 \right)}  +\frac{\TS}{2} \intO{ \vrh^{k-1}|D_t \vuh^k|^2  }
-\sum_{ \sigma \in \facesint } \intSh{ h^\eps   \jump{\vrh^k} \jump{ \frac{1}{2}|\vuh^k|^2    } }
\\& \quad
+\frac{1}{2}\sum_{ \sigma \in \facesint } \intSh{ (\vrh^k)^{\rm out} \left[\Ov{\vuh^k} \cdot \vc{n} \right]^-\jump{ \vuh^k }^2 }.
\end{split}
\]
Further, summing up the previous two observations we infer that
\begin{equation} \label{k1}
\begin{split}
& D_t \intO{ \frac{1}{2} \vrh |\vuh^k|^2  }
+  \mu   \norm{\Gradedge \vuh^k}_{L^2}  + (\mu +\lambda) \intO{  |\Divh \vuh^k|^2}
\\&
= \sum_{ \sigma \in \facesint } \intSh{ \Ov{p_h^k} \vc{n} \cdot \jump{ \vuh^k }  }
- \sum_{ \sigma \in \facesint } \intSh{ h^\eps  \jump{ \vrh^k \vuh^k }  \jump{\vuh^k  }  }
 +\sum_{ \sigma \in \facesint } \intSh{ h^\eps  \jump{\vrh^k} \jump{ \frac{1}{2}|\vuh^k|^2    } }
\\& \quad
- \frac{\TS}{2} \intO{ \vrh^{k-1}|D_t \vuh^k|^2  }
-\frac12 \sum_{ \sigma \in \facesint } \intSh{ (\vrh^k)^{\rm up} |\Ov{\vuh^k} \cdot \vc{n} |\jump{ \vuh^k } ^2   }
.\end{split}
\end{equation}
Finally, combining  \eqref{k1} with \eqref{r1} and using the equalities \eqref{basic_eq1}--\eqref{basic_eq2} we get
\begin{equation*}
\begin{split}
& D_t \intO{ \left(\frac{1}{2} \vrh^k |\vuh^k|^2  + \Hc(\vrh^k) \right)  }
+ h^\eps  \sum_{ \sigma \in \facesint } \intSh{  \Ov{\vrh^k} \jump{ \vuh^k}^2  }
+  \mu   \norm{\Gradedge \vuh^k}_{L^2}^2  + (\mu +\lambda) \intO{  |\Divh \vuh^k|^2}
\\& =
- \frac{\TS}2 \intO{ \Hc''(\xi)|D_t \vrh^k|^2  }
- \frac12 \sum_{ \sigma \in \facesint } \intSh{ \Hc''(\zeta) \jump{  \vrh^k } ^2 \left(h^\eps  + | \Ov{\vuh^k} \cdot \vc{n} | \right)}
\\&
\quad - \frac{\TS}{2} \intO{ \vrh^{k-1}|D_t \vuh^k|^2  }
- \frac12 \sum_{ \sigma \in \facesint } \intSh{ (\vrh^k)^{\rm up} |\Ov{\vuh^k} \cdot \vc{n} |\jump{ \vuh^k } ^2   } ,
\end{split}
\end{equation*}
which completes the proof.
\end{proof}

\subsection{Uniform bounds}
Having established all necessary ingredients, we are ready to discuss the available {\it a priori} bounds for solutions of scheme \eqref{scheme}.
From the total energy balance \eqref{energy_stability} and the Sobolev inequality \eqref{sobolev_inequ}, we directly get the estimates comprised in the following corollary.
\begin{Corollary}\label{ests}Let $(\vrh,\vuh)$ satisfy scheme \eqref{scheme} for $\gamma>1$.
Then the following estimates hold
\begin{subequations}\label{ests1}
\begin{align}
\label{est_ener}
\norm{\vrh \vuh^2}_{L^{\infty}L^1} & \lesssim 1,\\
\label{est_rho_lilg}
\norm{\vrh}_{L^{\infty}L^\gamma} & \lesssim 1,\\
\label{est_m}
 \norm{\vrh\vuh}_{L^\infty L^{\frac{2\gamma}{\gamma+1}}} & \lesssim 1,\\
 \label{est_gradu}
  \norm{\Gradedge \vuh}_{L^2 L^2}  &\lesssim 1,\\
\label{est_divu}
\norm{\Divh \vuh}_{L^{2}L^2}   &  \lesssim 1,  \\
\label{est_ul2l6}
\norm{\vuh}_{L^{2}L^6} & \lesssim 1,\\
\label{est_r_ujump}
h^\eps \int_0^T \sum_{ \sigma \in \facesint } \intSh{  \Ov{\vrh} \jump{ \vuh}^2}\dt  &  \lesssim 1, \\
\label{est_rjump}
 \int_0^T  \sum_{ \sigma \in \facesint } \intSh{ \Hc''(\zeta) \jump{  \vrh } ^2 (h^\eps + | \Ov{\vuh} \cdot \vc{n} |) } \dt &  \lesssim 1,
\end{align}
\end{subequations}
where $\zeta \in \co{\vr_K}{\vr_L}$ for any $\sigma=K|L \in \facesint$.
\end{Corollary}
 To show the consistency of the numerical scheme we shall need further bounds on the numerical solution, which can be derived provided the adiabatic coefficient in \eqref{pressure} lies in the physically realistic range $\gamma\in(1,2).$
\begin{Lemma}\label{ests2}
 Let $(\vrh, \vuh)$ satisfy scheme \eqref{scheme}, $h\in(0,1)$ and  $\gamma \in(1,2)$. Then there hold
 \begin{subequations}\label{est_rho_gamma12}
\begin{align}
%\label{est_srho_l2li}
%\norm{\sqrt{\vrh}}_{L^2 L^\infty }  &\aleq h ^{ - \frac{\eps+2}{2\gamma}} , \\
\label{est_rho_l2l2}
\norm{\vrh}_{L^2 L^2 }  &\aleq h ^{ - \frac{\eps+2}{2\gamma}}, \\
\label{est_m_l2l2}
 \norm{\vrh \vuh} _{L^2L^2} &\aleq h^{- \frac{\eps+2}{2\gamma}}.
\end{align}
\end{subequations}
\end{Lemma}
\begin{proof}
We start the proof by recalling the Sobolev inequality for the broken norm
\[ \norm{f_h}_{L^6}^2 \aleq \norm{f_h}_{L^2}^2 + \sum_{ \sigma \in \facesint } \intSh{ \frac{\jump{f_h}^2}{d_\sigma} }
=
\norm{f_h}_{L^2}^2 +  \norm{\Gradedge f_h}_{L^2}^2, \; \forall f_h\in Q_h,
\]
and the algebraic inequality
\[ a \gamma (\vr_L^{\gamma/2} -\vr_K^{\gamma/2})^2  \leq \frac{\pd^2 \Hc(z) }{\pd \vr^2} (\vr_L -\vr_K)^2,\; \forall\; z \in \co{\vr_L}{\vr_K}, \vr_L, \vr_K>0 \text{ if } \gamma \in(1,2).
\]
Then we indicate from the estimate of the density jumps \eqref{est_rjump} that
\[  \norm{\Gradedge \vrh^{\gamma/2}}_{L^2L^2}^2 =
\int_0^T \sum_{ \sigma \in \facesint } \intSh{ \frac{1}{d_\sigma} \jump{\vrh ^{\gamma/2}}^2} \aleq h^{-(\eps+1)}.
 \]
Applying the above inequalities, the inverse estimate and the estimate \eqref{est_rho_lilg} we derive
\begin{equation*}
\begin{aligned}
\norm{\vrh}_{L^1 L^\infty } &
= \int_0^T \norm{\vrh ^{\gamma/2}}_{ L^\infty }^{2/\gamma}  \dt
\leq \int_0^T \left( h^{-1/2} \norm{\vrh ^{\gamma/2}}_{ L^6 } \right) ^{2/\gamma}  \dt \\
&\leq h^{-1/\gamma} \int_0^T \left( \norm{\vrh ^{\gamma/2}}_{ L^2 }^2
+ \norm{\Gradedge \vrh^{\gamma/2}}_{L^2} ^2  \right)^{1/\gamma}  \dt
\leq h^{-1/\gamma} \left( \norm{\vrh}_{ L^1 L^\gamma } + \norm{\Gradedge \vrh^{\gamma/2}}_{L^{\gamma/2}L^2}  ^{2/\gamma}  \right)\\
&\leq h^{-1/\gamma} \left( \norm{\vrh}_{ L^1 L^\gamma } + \norm{\Gradedge \vrh^{\gamma/2}}_{L^2L^2}  ^{2/\gamma}  \right)
\leq  h ^{ - \frac{\eps+2}{\gamma}}.
\end{aligned}
\end{equation*}
Further application of the above inequality together with the Gagliardo-Nirenberg interpolation inequality, H\"older's inequality, and the density estimate \eqref{est_rho_lilg}  immediately yield \eqref{est_rho_l2l2}, i.e.,
\begin{equation*}
 \norm{\vrh}_{L^2 L^2 } = \left( \int_0^T \norm{\vrh}_{L^2}^2 \dt  \right)^{1/2}
 \leq \left( \int_0^T \norm{\vrh}_{L^1}\norm{\vrh}_{L^\infty } \dt  \right) ^{1/2}
   \leq \norm{\vrh}_{L^\infty L^1}^{1/2} \norm{\vrh}_{L^1 L^\infty }^{1/2}
   \aleq  h ^{ - \frac{\eps+2}{2\gamma}}.
\end{equation*}
Finally, the estimate \eqref{est_m_l2l2} can be shown in the following way
\begin{equation*}
 \norm{\vrh \vuh} _{L^2L^2} \aleq
\norm{ \sqrt{\vrh} }_{L^2 L^\infty} \norm{ \sqrt{\vrh} \vuh} _{L^\infty L^2}
=\norm{\vrh}_{L^1 L^\infty }^{1/2}  \norm{ \vrh \vuh^2} _{L^\infty L^1} ^{1/2}
\aleq h^{- \frac{\eps+2}{2\gamma}}.
\end{equation*}
\end{proof}

%%%%%%%%%%%%%%%%%%%%%%%%%%%%%%%%%%%% ORIGINAL %%%%%%%
\section{Consistency}
\label{sec_Consistency}

 Next step towards the convergence of the approximate solutions is the consistency of the numerical scheme. In particular, we require the numerical solution to satisfy the weak formulation of the  continuous problem up to a residual term  vanishing for $h \to 0.$
\begin{Theorem} \label{Tm2}
Let $(\vrh, \bfu_h)$ be a solution of the approximate problem \eqref{scheme} on the time interval $[0,T]$ with $\TS\approx h$, $1<\gamma <2$ and $0<\eps <\min\left\{1, 2(\gamma-1)\right\}$.
Then
\begin{equation} \label{cP1}
- \intO{ \vrh^0 \phi(0,\cdot) }  =
\int_0^\tau \intO{ \left[ \vrh \partial_t \phi + \vrh \vuh \cdot \Grad \phi \right]} \dt  + \int_0^T
e_{1,h} (t, \phi) \dt,
\end{equation}
for any $\phi \in C_c^3([0,T) \times \Ov{\Omega})$;
\begin{equation} \label{cP2}
\begin{split}
- &\intO{ \vrh^0 \vuh^0 \bfphi(0,\cdot) }  =
\int_0^T \intO{ \left[ \vrh \vuh \cdot \partial_t \bfphi + \vrh \vuh \otimes \vuh  : \Grad \bfphi  + p_h \Div \bfphi \right]} \dt,
\\&
 -  \mu \int_0^T \intO{  \Gradedge \vuh : \Grad \bfphi}  \dt
- (\mu + \lambda)\int_0^T \intO{ \Divh \vuh\, \Div \bfphi}\dt
 + \int_0^T e_{2,h} (t, \bfphi) \dt
\end{split}
\end{equation}
for any $\bfphi \in C^3_c([0,T] \times {\Omega}; R^d)$;
\[
\| e_{j,h} (\cdot, \phi ) \|_{L^1(0,T)} \lesssim h^\beta  \left( \| \phi \|_{C^2} + h \| \phi \|_{C^3} \right), \, j=1,2,  \, \mbox{ for some }\ \beta > 0.
\]
\end{Theorem}

\begin{proof}
Let $\phi \in C_c^{3}([0,T)\times \overline{\Omega})$ and $\bfphi \in C_c^{3}([0,T)\times \Omega;R^d).$  We test the   equations \eqref{scheme_den} and \eqref{scheme_mom} with $\Pim \phi $ and $\Pim \bfphi $, respectively, and deal with each term separately.
\paragraph{Step 1 -- time derivative terms:}
\begin{equation*}
\begin{aligned}
&\int_0^T \intO{ D_t r_h \Pim \phi } \dt =
\int_0^T \intO{ \frac{r_h(t)- r_h(t-\Delta t) } {\Delta t} \phi(t) } \dt
\\ &
= \frac{1}{\TS}\int_0^T \intO{ r_h(t) \phi(t) } \dt -
\frac{1}{\Delta t}\int_{-\Delta t}^{T-\Delta t} \intO{ r_h(t) \phi(t+ \Delta t) } \dt
\\ &=  - \int_0^T \intO{ r_h(t) D_t \phi(t) } \dt
+ \frac{1}{\Delta t}\int_{T-\Delta t}^{T} \intO{ r_h(t) \phi(t+ \Delta t) } \dt
 - \frac{1}{\Delta t}\int_{-\Delta t}^{0} \intO{ r_h(t) \phi(t+ \Delta t) } \dt
\\ & = - \int_0^T \intO{ r_h(t) D_t \phi (t) } \dt - \intO{ r_h^0 \phi (0) }
\\ &  = - \int_0^T \intO{ r_h(t) \pd_t \phi (t) } \dt - \intO{ r_h^0 \phi (0) } + \TS  \norm{\phi}_{C^2} \norm{r_h}_{L^1L^1}
,
\end{aligned}
\end{equation*}
where  $r_h$ stands for $\vrh$ or $\vrh u_{i,h},$ $ i=1,\ldots,d$.
Recalling the estimates \eqref{est_rho_lilg} and \eqref{est_m} we know that
\[ \norm{\vrh}_{L^1L^1} \aleq \norm{\vrh}_{L^\infty L^\gamma} \aleq 1 \quad \mbox{and}\quad
\norm{\vrh \vuh }_{L^1L^1} \aleq \norm{\vrh \vuh}_{L^\infty L^{\frac{2\gamma}{\gamma+1}}} \aleq 1.
\]
Thus, we have
\begin{subequations}\label{consistency_timederivative}
\begin{align}
\int_0^T \intO{ D_t \vrh \Pim\phi} \dt
& + \int_0^T \intO{ \vrh(t) \pd_t \phi (t) } \dt
+ \intO{ \vrh^0 \phi (0) }  \aleq \TS, \\
\int_0^T \intO{ D_t (\vrh \vuh) \Pim \bfphi } \dt
& + \int_0^T \intO{ \vrh(t)\vuh(t)  \pd_t \bfphi(t) } \dt
+ \intO{ \vrh^0 \vuh^0 \bfphi(0) }  \aleq \TS,
\end{align}
for the continuity and the momentum equations, respectively.
\end{subequations}

\paragraph{Step 2 -- convective terms:}\hspace{1ex} \\
To deal with the convective terms, it is convenient to
recall Lemma \ref{lem_convective_trans}:
\begin{align*}
\int_0^T \intO{  r_h \vuh \cdot \Grad \phi  } \dt  - \int_0^T \sum_{\sigma \in \facesint} \intSh{ F[ r_h,\vuh ] \jump{  \Pim \phi}}  \dt =\sum_{j=1}^4 E_j(r_h),
\end{align*}
where
\begin{equation*}
\begin{aligned}
& E_1(r_h)= \frac12 \int_0^T \sum_{\sigma \in \facesint} \intSh{
  |\Ov{\vuh} \cdot \vc{n}| \jump{r_h }  \jump{ \Pim \phi}  }  \dt   ,
\\& E_2(r_h)= \frac14  \int_0^T \sum_{\sigma \in \facesint} \intSh{
 \jump{\vuh} \cdot \vc{n}   \jump{r_h }  \jump{  \Pim \phi}  }  \dt ,
\\& E_3(r_h)= \int_0^T\sum_{\sigma \in \facesint} \intSh{ h^\eps   \jump{r_h }   \jump{  \Pim \phi} } \dt ,
\\& E_4(r_h)= \int_0^T \intO{  r_h \vuh \cdot \Big(\Grad \phi - \Gradmesh \Pid  \big( \Pim \phi\big) \Big) } \dt,
\end{aligned}
\end{equation*}
are the error terms to be estimated. Again, $r_h$ is either $\vrh$ or $\vrh u_{i,h},$ $i=1,\ldots,d.$

\begin{itemize}[wide=0pt]
\item  Firstly, for the error term $E_1$ we can write

\begin{equation*}
\begin{aligned}
E_1(r_h)= &\frac12 \int_0^T \sum_{\sigma \in \facesint} \intSh{ |\Ov{\vuh}\cdot \vc{n}| \jump{r_h }  \jump{ \Pim \phi} }  \dt
= \frac12 \int_0^T \sum_{\sigma \in \facesint} \intSh{ |\Ov{u_{i,h}}| \jump{r_h }  \jump{ \Pim \phi} }  \dt
\\
&=  \frac12 \int_0^T \sumi \sum_{\sigma \in \facesinti} \int_{D_\sigma} h_i |\Ov{u_{i,h}}|    \pdedgei r_h  \pdedgei \Pim \phi  \dx \dt
\\&=
- \frac{1}{2} \int_0^T \sumi \sum_{K \in \grid} \int_K h_i r_h \pdmeshi \left(  |\Ov{u_{i,h}}|     \pdedgei  \Pim \phi
\right)
\dx  \dt
 \\& =
-\frac{1}{2} \int_0^T  \sumi \sum_{K \in \grid} \int_K  r_K h_i \bigg(  \Pim |\Ov{u_{i,h}}| \; \pdmeshi (\pdedgei \Pim\phi) + \left(\pdmeshi |\Ov{u_{i,h}}| \right)    \Pim (\pdedgei \Pim\phi) \bigg)\dx  \dt,
\end{aligned}
\end{equation*}
where we have used the integration by parts formula \eqref{int_by_part_gradm}, the product rule
\[r_2q_2 -r_1q_1 = \frac{r_1+ r_2}{2} (q_2 -q_1)  + \frac{q_1 +q_2}{2} (r_2- r_1).\]
 Further, employing the inequality $\left(\frac{a+b}{2} \right)^2 \leq \frac{a^2 +b^2}{2}$ twice, we claim  $\norm{\Pim|\Ov{u_{i,h}}|}_{L^2} \aleq \norm{u_{i,h}}_{L^2}$. Similarly, we claim
$ \norm {\pdmeshi \Ov{u_{i,h}}} _{L^2}  \aleq \norm{\pdedgei u_{i,h}}_{L^2} $ as
 $ \left(\pdmeshi \Ov{u_{i,h}}\right)_K =\Pim \left( \pdedgei u_{i,h} \right) _K $.
Then applying H\"older's inequality, interpolation error estimates \eqref{n4c2_edge}, \eqref{n4c3}, the velocity estimates \eqref{est_gradu}, \eqref{est_ul2l6}, the fact $|\pd_x u_i| \geq \pd_x|u_i|$, and noticing $\Delta_h^{(i)} r := \pdmeshi \pdedgei r$,  we derive
\begin{equation*}
\begin{aligned}
E_1(r_h)& =
-\frac{1}{2}
 \int_0^T  \sumi \sum_{K \in \grid} \int_K  r_K h_i \bigg(  \Pim |\Ov{u_{i,h}}| \; \pdmeshi (\pdedgei \Pim\phi) + \left(\pdmeshi |\Ov{u_{i,h}}| \right)    \Pim (\pdedgei \Pim\phi) \bigg)\dx  \dt
\\&
\aleq \sumi h_i \left(\int_0^T \sum_K \int _K r_K^2 \right)^{1/2}
\Bigg[ \left(\int_0^T \sum_{K\in\grid} \int _K \left( \Pim |\Ov{u_{i,h}}| \right)_K ^2 \right)^{1/2}
 \norm{ \Delta_h^{(i)} \Pim \phi} _{L^\infty L^\infty}
 \\& \qquad \qquad \qquad \qquad \qquad\qquad  \quad +
 \left(\int_0^T \sum_{K\in\grid} \int _K \left(\pdmeshi \Ov{u_{i,h}} \right)  ^2 \right)^{1/2}
 \norm{ \Pim( \pdedgei \Pim\phi) } _{L^\infty L^\infty}
 \Bigg]
 \\ & \aleq
h \sumi  \norm{r_h}_{L^2 L^{2}}  \left(  \norm{\Delta_h^{(i)} \Pim \phi}_{L^\infty L^\infty} \norm{  u_{i,h} }_{L^2 L^2}
+  \norm{\pdedgei  \Pim \phi}_{L^\infty L^\infty} \norm{  \pdedgei u_{i,h}  }_{L^2 L^2} \right)
\\& \aleq
 h \norm{r_h}_{L^2 L^{2}}  \left(  \norm{\Delta_h  \Pim \phi}_{L^\infty L^\infty} \norm{  \vuh }_{L^2 L^2}
+  \norm{\Gradedge  \Pim \phi}_{L^\infty L^\infty} \norm{  \Gradedge \vuh  }_{L^2 L^2} \right)
\\&
\aleq  h \norm{r_h}_{L^2 L^{2}}.
\end{aligned}
\end{equation*}
Consequently, applying  the density estimate \eqref{est_rho_l2l2}, and the momentum estimate \eqref{est_m_l2l2} indicates
\begin{equation*}
E_1(r_h) \lesssim h^\beta, \quad \beta = 1- \frac{\eps+2}{2\gamma} >0,  \mbox{ provided } \eps<2(\gamma -1),
\end{equation*}
for $r_h$ being $\vrh$ or $\vrh u_{i,h},$ $i=1,\ldots,d$.

\item Secondly, we deal with the error term $E_2$. In accordance with \eqref{n4c}, we have
\begin{equation*}
E_2(r_h) \aleq h \sum_{\sigma \in \facesint} \intSh{ |\jump{\vuh}\cdot \vc{n} \jump{r_h }|  }  \dt .
\end{equation*}
For $r_h$ being $\vrh$, we further  write
\begin{equation*}
\begin{aligned}
E_2(\vrh) &  \aleq
 h  \left( \int_0^T \sum_{\sigma \in \facesint} \intSh{ \jump{\vuh}^2 }  \dt  \right)^{1/2}  \left( \int_0^T \sum_{\sigma \in \facesint} \intSh{ \jump{\vrh}^2 }  \dt  \right)^{1/2}
\\& \aleq h h^{1/2} \left( \int_0^T \sum_{\sigma \in \facesint} \intSh{ \Ov{\vrh}^2 }  \dt  \right)^{1/2}
\\&
\aleq h^{3/2} h^{-1/2} \norm{\vrh}_{L^2L^2}
\aleq h^ \beta, \quad \beta= 1-\frac{\eps+2}{2\gamma}>0, \mbox{ as soon as }\eps < 2(\gamma-1).
\end{aligned}
\end{equation*}
Here we have used H\"older's inequality, \eqref{est_gradu}, \eqref{est_rho_l2l2}, and the fact $|\jump{\vrh}| < 2 \Ov{\vrh}$.

For $r_h$ being $\vrh u_{i,h}$, we get
\begin{equation*}
E_2(\vrh \vuh) \aleq
 h   \int_0^T \sum_{\sigma \in \facesint} \intSh{ |\jump{\vuh}\cdot \vc{n}|\; \left| \jump{\vrh } \Ov{\vuh}   +  \jump{\vuh } \Ov{\vrh}  \right| }  \dt :=T_1 +T_2 .
\end{equation*}
To control the residual term $T_1$ we apply H\"older's inequality, \eqref{est_ener}, \eqref{est_r_ujump}, inverse estimate \eqref{inverse_inequality} and the inequality $|\jump{\vrh}| < 2 \Ov{\vrh}$ to obtain
\begin{equation*}
\begin{aligned}
T_1 & \aleq
h \int_0^T \sum_{\sigma \in \facesint} \intSh{ |\jump{\vuh}\cdot \vc{n}|  \Ov{\vrh } |\Ov{\vuh}| }  \dt
  \\& \aleq
 h \left( \int_0^T \sum_{\sigma \in \facesint} \intSh{  \Ov{\vrh } \jump{\vuh}^2 } \right)^{1/2}
 \left( \int_0^T \sum_{\sigma \in \facesint} \intSh{  \Ov{\vrh } |\Ov{\vuh}|^2 } \right)^{1/2}
\\ & \aleq
 h^{(1-\eps)/2} .
\end{aligned}
\end{equation*}
Further, applying \eqref{est_r_ujump} we can control the residual term $T_2$ as
\begin{equation*}
T_2  =
 h   \int_0^T \sum_{\sigma \in \facesint} \intSh{ |\jump{\vuh}\cdot \vc{n} | \jump{\vuh }| \Ov{\vrh}   }  \dt   \aleq  h^{1-\eps}.
\end{equation*}
Therefore, we claim that provided $\eps < 2(\gamma-1)$ we have
\[E_2(r_h) \lesssim h^\beta, \quad \beta >0
\]
for $r_h$ being $\vrh$ or $\vrh u_{i,h},$ $i=1,\ldots,d$.

\item Next, we  consider the error term $E_3.$ Analogously as above, the  integration by parts formula \eqref{int_by_part_las}, H\"older's inequality, and the interpolation error \eqref{n4c3} yield
\begin{align*}
 E_3(r_h) &= h^\eps \int_0^T\sum_{\sigma \in \facesint} \intSh{  \jump{r_h } \jump{\Pim\phi }   } \dt
=-h^{\eps+1}\int_0^T\intO{r_h \Delta_h \Pim\phi }\dt
\\&
\lesssim h^{\eps+1}  \norm{r_h}_{L^1L^1}  \left(\norm{\phi }_{C^2} + h \norm{\phi }_{C^3}\right)
\lesssim h^{\eps+1}  \norm{r_h}_{L^1L^1}.
\end{align*}
  Using the estimates \eqref{est_rho_lilg} and \eqref{est_m}
we can conclude  for $r_h$ being $\vrh$ or $\vrh u_{i,h},$ $i=1,\ldots,d,$ that
\begin{align*}
& E_3(r_h)\lesssim h^{\eps+1}.
\end{align*}

\item Finally, using the estimates of kinetic energy \eqref{est_ener} and momentum \eqref{est_m} together with the interpolation error \eqref{n4c2_mesh} we obtain for $r_h$ being $\vrh$ or $\vrh u_{i,h},$ $i=1,\ldots,d$ that
\begin{align*}
 E_4(r_h) &= \int_0^T \intO{  r_h \vuh \cdot \left(\Grad \phi - \Gradmesh \Pid \big(\Pim\phi\big)\right) } \dt
\lesssim h \norm{\phi }_{C^2} \norm{  r_h \vuh  }_{L^1 L^1}
 \lesssim h  \norm{ r_h \vuh }_{L^\infty L^1}  \lesssim h.
\end{align*}
Consequently, we conclude the  consistency formulation of the convective terms in both equations \eqref{scheme_den} and \eqref{scheme_mom}, by collecting the above estimates of the four terms $E_j, \ j=1,\ldots, 4,$
\begin{subequations}\label{consistency_convective}
\begin{align}
&\intO{  \vrh \vuh \cdot \Grad \phi  } - \sum_{\sigma \in \facesint} \intSh{ F[ \vrh,\vuh ] \jump{ \Pim\phi}} \lesssim h^{\beta_1},
\\
&\intO{  \vrh \vuh\otimes\vuh : \Grad \bfphi } - \sum_{\sigma \in \facesint} \intSh{ F[ \vrh \vuh,\vuh ] \jump{ \Pim  \bfphi }} \lesssim h^{\beta_2},\label{cons_mom_convective}
\end{align}
\end{subequations}
for some $\beta_1,$ $\beta_2>0$ provided $0< \;  \eps < \min\{1, 2(\gamma-1)\}$.
\end{itemize}

\begin{Remark}
We would like to emphasize that the additional numerical diffusion of order $\mathcal{O}(h^{\eps+1})$ plays a crucial role in order to obtain  the consistency of the convective terms, which requires $\eps < \min\{1,2(\gamma-1)\}$. Indeed the deriviation of \eqref{consistency_convective} requires \eqref{est_r_ujump} and  uses Lemma \ref{lem_r1}, which benefits from the uniform bounds \eqref{est_rjump} of the $h^\eps$-terms.
\end{Remark}

\paragraph{Step 3 -- viscosity terms:}\hspace{1ex} \\
 In accordance with (\ref{n4c2_edge}) and \eqref{est_gradu}  we can control the viscosity terms. Indeed, we have
\begin{subequations}\label{consistency_viscosity}
\begin{equation}
\begin{split}
& \int_0^T \intO{ \Gradedge \vuh : \Grad \bfphi}\dt  -  \int_0^T \sum_{\sigma \in \facesint} \intSh{ \frac{1}{d_\sigma}\jump{ \vuh}  \cdot \jump{ \Pim  \bfphi   }  }  \dt
\\&=
\int_0^T \intO{ \Gradedge \vuh :( \Grad \bfphi - \Gradedge \Pim \bfphi   ) }\dt
\lesssim
\norm{\Gradedge \vuh}_{L^2L^2} h \norm{\bfphi}_{C^2}
\lesssim h,
\end{split}
\end{equation}
and for the divergence term we get
\begin{equation}
\begin{split}
& \int_0^T \intO{ \Divh \vuh\, \Divh \big(\Pim \bfphi\big)  } -\int_0^T \intO{\Divh \vuh \, \Div \bfphi}\dt
\\&=
\int_0^T \intO{\Divh \vuh \, \Big(\Divh \big(\Pim \bfphi\big)   -\Div \bfphi\Big) }\dt
\lesssim
\norm{\Divh \vuh}_{L^2L^2} h \norm{\bfphi}_{C^2}
\lesssim h,
\end{split}
\end{equation}
by using  \eqref{est_divu} and \eqref{n4c2_mesh}.
\end{subequations}
\paragraph{Step 4 -- pressure term:}\hspace{1ex} \\
The pressure term can be controlled by using the integration by parts formula \eqref{basic_eq2},
% definition of discrete divergence \eqref{div_mesh},
the interpolation error \eqref{n4c2_mesh}, and the estimate \eqref{est_rho_lilg}, i.e.,
\begin{equation}\label{consistency_pressure}
\begin{split}
\int_0^T  \sum_{\sigma \in \facesint}&  \intSh{ \Ov{p}_h \vc{n} \cdot \jump{ \Pim  \bfphi   }  } \dt
-   \int_0^T \intO{ p_h \Div \bfphi } \dt \\
&=- \int_0^T \sum_{\sigma \in \facesint} \intSh{ \Ov{\Pim  \bfphi } \cdot \vc{n} \jump{p_h }  } \dt
-  \int_0^T \sum_{K \in \grid} \int_K p_h \Div \bfphi \dx \dt
\\& = \int_0^T  \sum_{K \in \grid} p_K \sum_{\sigma \in \facesK} \int_{\sigma}  \Ov{\Pim  \bfphi  } \cdot \vc{n} \ds \dt  -  \int_0^T \sum_{K \in \grid} \int_K p_h \Div \bfphi \dx \dt
\\
&= \int_0^T  \sum_{K \in \grid} \int_K p_h \left( \Divh \big(\Pim  \bfphi \big) -  \Div \bfphi \right) \dx \dt
\lesssim
\norm{p_h}_{L^\infty L^1} h \norm{\bfphi}_{C^2}
\lesssim h.
\end{split}
\end{equation}
Collecting the inequalities \eqref{consistency_timederivative}--\eqref{consistency_pressure} we complete the proof of Theorem \ref{Tm2}.
\end{proof}

\section{Convergence}\label{sec_Convergence}
In this section, we show the main result, the convergence of the numerical solution to the strong solution of the system \eqref{ns_eqs} on the
lifespan of the latter.
To this end we start by introducing the concept of the dissipative measure-valued (DMV) solutions to (\ref{ns_eqs}). The interested reader may consult~\cite{FGSWW} for the discussion about the concept of DMV solutions and the DMV--strong uniqueness principle that will be used later in this section.
\begin{Definition}[DMV solution]\label{def_dmvs}
We say that a parametrized family of probability measures $\{ \Nu _{t,x} \}_{(t,x)\in (0,T)\times\Omega}$,
\begin{equation*}
\Nu _{t,x} \in L^{\infty}_{weak} \Big((0,T)\times\Omega;\, \mathcal{P}(Q) \Big),\  Q = \left\{ [\vr, \vu] \ \Big|\ \vr \in [0, \infty), \ \vu \in R^N \right\},
\end{equation*}
is \emph{a dissipative measure-valued (DMV) solution} of the Navier--Stokes system in $(0,T)\times\Omega$, with the initial condition $\Nu_{0,x}\in \mathcal{P}(Q)$ and dissipation defect $\mathcal{D} \in L^{\infty}(0,T),\ \mathcal{D} \ge 0$, if the following holds:
\begin{itemize}
\item
\begin{equation*}
\left[\int_{\Omega} \myangle{\Nu_{t,x};\vr} \phi(t,\cdot)\dx\right]_{t=0}^{t=\tau} =
\int_0^{\tau}  \int_{\Omega}[\myangle{\Nu_{t,x};\vr} \pd_t \phi + \myangle{\Nu_{t,x};\vr \bfu}\cdot \nabla_x \phi] \dx \dt
\end{equation*}
for any $0 \leq \tau \leq T$ and  $\phi \in C^1\big([0,T]\times \Omega\big)$;

\item
\begin{multline*}%\label{mvs-mome}
\left[\int_{\Omega} \myangle{\Nu_{t,x};\vr \bfu} \bfphi(t,\cdot)\dx \right]_{t=0}^{t=\tau}  =
\int_0^{\tau}  \int_{\Omega}[\myangle{\Nu_{t,x};\vr \bfu} \pd_t \bfphi + \myangle{\Nu_{t,x};\vr \bfu \otimes \bfu + p(\vr) \I} : \nabla_x \bfphi  ] \dx \dt
\\
- \int_0^{\tau} \int_{\Omega} \mathcal{S}(\Grad \bfu): \Grad \bfphi \dx \dt
+ \int_0^{\tau} \myangle{r^M;\nabla_x \bfphi }\dt
\end{multline*}
for any $0 \leq \tau \leq T$ and  $\bfphi \in C^1_c\big([0,T]\times\Omega; R^d\big)$,
where
\[
 \vu = \left< \Nu_{t,x} ; \vu \right>,\ \vu \in L^2(0,T; W^{1,2}(\Omega; R^d)),
\]
\[
\mathcal{S}(\Grad \bfu) = \mu (\Grad \vu + \Grad^t \vu)  +  \lambda \Div \vu \I,
\text{ and }
r^M\in L^1\big(0,T;\mathcal{M}({\Omega})\big);
\]

\item
\begin{equation*}%\label{mvs-entroy}
\left[\int_{\Omega} \myangle{\Nu_{t,x};\frac12 \vr \bfu^2+\Hc(\vr)} \dx \right]_{t=0}^{t=\tau}  + \int_0^{\tau} \int_{\Omega} \mathcal{S}(\Grad \bfu):\nabla_x\bfu \dx \dt + \mathcal{D}(\tau) \le 0,
\end{equation*}
for a.a. $0 \leq \tau \leq T$. The dissipation defect $\mathcal{D}$ dominates the concentration measure $r^M$, specifically,
\begin{equation*}
\left|\myangle{r^M(\tau); \phi }\right| \lesssim \xi(\tau) \mathcal{D}(\tau) \norm{\phi }_{C({\Omega})}, \mbox{ for some}\ \xi \in L^{1}(0,T).
 \end{equation*}
%\item
%
%In the case of the no--slip boundary conditions (\ref{noslip}), $\mathcal{D}$ satisfies a variant of Poincar\' e inequality,
%\begin{equation} \label{PIC}
%\int_0^\tau \int_{\Omega} \myangle{\Nu_{t,x}; |\vc{v} - \vc{u}|^2 } \dx {\rm d}t\leq c_p \mathcal{D}(\tau).
%\end{equation}
\end{itemize}

\end{Definition}

%
%Recently, the approach to convergence proof via the (DMV) solutions has been successfully applied to the two finite difference Marker--And--Cell schemes, see~\cite{MiShe18, NoShe18}.

\subsection{Convergence to dissipative measure-valued solution}
In this subsection, we show that any Young measure generated by a family of numerical solutions is a DMV solution in the sense of Definition in~\ref{def_dmvs}.

\begin{Theorem}\label{thm_dmvs}
Let $\{(\vrh^k,\bfu_h^k)\}_{k=1}^{N_T}$ be a family of solutions
generated by the numerical scheme \eqref{scheme}, with $\Delta t \approx h$, $1<\gamma <2$, $0<\eps <\min\{1, 2(\gamma-1)\}$, and the initial data satisfying
%\[\vr_0 \in L^\infty (0,T; L^\gamma(\Omega)), \ \vr_0>0, \vu_0 \in L^2 ((0,T)\times\Omega; R^d).\]
\[
\vr_0  \in  L^\gamma(\Omega), \ \vr_0>0,\  \vu_0 \in L^2 (\Omega; R^d).
\]
Then any Young measure $\{ {\Nu_{t,x}}\}_{(t,x)\in(0,T)\times\Omega} $ generated by $(\vrh^k,\bfu_h^k)$ for $h\rightarrow 0$ represents a dissipative measure-valued solution of the Navier--Stokes system (\ref{ns_eqs}) in the sense of Definition~\ref{def_dmvs}.
\end{Theorem}
\begin{proof}
%{\bf Step 1}: Weak limit.

We may use the energy estimates (\ref{theorem_stability}) to deduce that, at least for suitable subsequences,
\[
\begin{split}
 \vrh &\to \vr \ \mbox{weakly-(*) in}\ L^\infty(0,T; L^\gamma(\Omega)),\ \vr \geq 0 \\
 \vuh &\to \bfu \ \mbox{weakly in}\ L^2((0,T) \times \Omega; R^d),\\
 \mbox{where} \ \bu  & \in L^2(0,T; W^{1,2}(\Omega)), \ \Gradedge \vuh \to \nabla _x \bu \ \mbox{weakly in} \ L^2((0,T) \times \Omega;
R^{d \times d}), \\
\vrh \vuh  &\to \widetilde{\vr \bu} \ \  \mbox{weakly-(*) in}\ L^\infty(0,T; L^{\frac{2\gamma}{\gamma + 1}}(\Omega; R^d)).
\end{split}
\]
where the superscript  `$\sim$' denotes the $L^1$-weak limit.

Note that, the limit functions satisfy the equation of continuity in the form
\begin{equation*} \label{M2}
- \intO{ \vr_0 \phi  (0, \cdot) } = \int_0^T \intO{ \left[  \vr \partial_t \phi  + \widetilde{\vr \bu} \cdot \nabla _x \phi  \right] } \dt, \;
\mbox{ for all } \phi  \in \DC([0, \infty) \times\Omega ),
\end{equation*}
which can be further rewritten as
\begin{equation} \label{M3}
\left[ \intO{ \vr \phi  (t, \cdot) } \right]_{t = 0}^{t = \tau} = \int_0^\tau
\intO{ \left[  \vr \partial_t \phi  + \widetilde{\vr \bu} \cdot \nabla _x \phi  \right] } \dt
\end{equation}
for any $0 \leq \tau \leq T$ and any $\phi  \in C^\infty([0,T] \times  \Omega)$.

%\noindent{\bf Step 2}: Young measure generated by numerical solutions.

 In accordance with the weak convergence statement derived in the preceding part, the family $[\vrh, \vuh]$ generates a Young measure - a parameterized measure~\cite{BALL2,PED1}
\[
\Nu_{t,x} \in L^\infty((0,T) \times \Omega; \mathcal{P}([0, \infty) \times R^d)) \ \mbox{for a.e.}\ (t,x) \in (0,T) \times \Omega,\quad \mbox{with}\quad \Nu_{0,x} = \delta_{[\vr_0(x), \bu_0(x)]},
\]
such that
\[
\left< \Nu_{t,x}, g(\vr, \bu) \right> = \widetilde{g(\vr, \bu)}(t,x)\ \mbox{for a.e.}\ (t,x) \in (0,T) \times \Omega,
\]
for any $g \in C([0, \infty) \times R^d)$ such that
\[
g(\vrh, \vuh) \to \widetilde{g(\vr, \bu)} \ \mbox{weakly in}\ L^1((0,T) \times \Omega).
\]
Accordingly, the equation of continuity (\ref{M3}) can be written as
\begin{equation} \label{M4}
\left[ \intO{ \vr \phi  (t, \cdot) } \right]_{t = 0}^{t = \tau} = \int_0^\tau
\intO{ \left[  \vr \partial_t \phi  + \left< \Nu_{t,x}, \vr \bu \right>  \cdot \nabla _x \phi  \right] } \dt .
\end{equation}

For the consistency formulation of the momentum equation (\ref{cP2}), we apply a similar treatment. Whence letting $h \to 0$ in (\ref{cP2}) gives rise to
\begin{equation} \label{M6}
\begin{split}
&\left[ \intO{ \left< \Nu_{t,x} ; \vr \bu \right> \cdot  \vcg{\phi }(t, \cdot) } \right]_{t = 0}^{t = \tau} = \int_0^\tau \intO{
\Big[ \left< \Nu_{t,x}; \vr \bu \right> \cdot \partial_t \vcg{\phi } + \left< \Nu_{t,x}; \vr \bu \otimes \bu + p(\vr)\I \right> : \nabla _x \vcg{\phi }  \Big] } \ \dt
\\
\quad & - \int_0^\tau \intO{ \Big[ \mu \Grad \bu : \nabla _x \vcg{\phi } + (\mu + \lambda) \Div \bu \cdot \Div \vcg{\phi } \Big] } \dt
+ \int_0^\tau \intO{ r^M : \nabla _x \bfphi } \dt
\end{split}
\end{equation}
for any $0 \leq \tau \leq T$, $\bfphi  \in \DC([0,T] \times \Omega; R^d)$ where the \emph{concentration remainder} reads
\[
r^M = \left\{ \vr \bu \otimes \bu + p(\vr) \mathbb{I} \right\} - \left< \Nu_{t,x}; \vr \bu \otimes \bu + p(\vr) \mathbb{I} \right>
\in [ L^\infty(0,T; \mathcal{M}(\Omega)) ]^{d \times d}.
\]
Hereafter $\{ s \}$ denotes the weak-(*) limit of a sequence $s_h$ in $\mathcal M([0,T]\times\Omega)).$

Similarly, letting $h \to 0$ in the energy inequality (\ref{energy_stability}) yields
\begin{equation} \label{M7}
\begin{split}
\left[ \intO{  \left< \Nu_{t,x}; \frac{1}{2}\vr | {\bu } |^2 + \Hc(\vr) \right>  } \right]_{t = 0}^{t = \tau} +
\int_0^\tau \intO{  \left( \mu |\Grad \bu |^2 + (\mu + \lambda) |\Div \bu |^2 \right) } \dt  + \mathcal{D}(\tau) &\leq 0
\end{split}
\end{equation}
for a.e. $\tau \in [0,T]$, with the  \emph{dissipation defect}
\begin{align*}
\mathcal D(\tau)&= \intO{\left\{\frac 1 2 \vr | {\bu } |^2 + \Hc(\vr)\right\}}-\intO{\left< \Nu_{\tau,x};  \frac 1 2 \vr | {\bu } |^2 + \Hc(\vr)\right>}\\
& + \int_0^{\tau}\intO{\left\{\mu|\Grad\bu|^2+(\mu+\lambda)|\Div\bu|^2\right\}}\dt - \int_0^{\tau}\intO{\mu|\Grad\bu|^2+(\mu+\lambda)|\Div\bu|^2}\dt
\end{align*}
 satisfying
\begin{equation} \label{M8}
\mathcal{D}(\tau) \geq  \liminf_{h \to 0} \left( \int_0^\tau  \norm{\Gradedge\vuh}_{L^2}^2  \dt \right) - \int_0^\tau \intO{ |\nabla _x \bu |^2 } \dt.
\end{equation}
Note that $\mathcal D \geq 0$ is a consequence of \cite[Lemma~2.1]{FGSWW} and the non--negativity of the total energy and $\mathcal S(\Grad\bu):\Grad\bu.$ Since  $|\vr_h\bu_h\otimes\bu_h + p(\vr_h)\mathbb I| \lesssim \vr_h|\bu_h|^2+\mathcal{H}(\vr_h),$ we again recall \cite[Lemma~2.1 ]{FGSWW} for
$$F(\vr,\bu)= | \vr \bu\otimes\bu + p(\vr)\mathbb I|, \quad \mbox{and} \quad  G(\vr,\bu)=\frac 1 2\vr |\bu|^2+\mathcal{H}(\vr)$$
to conclude
\begin{equation} \label{M9}
%\int_0^\tau \| r^M \|_{\mathcal{M}(\Omega)} \dt \le  \int_0^\tau \mathcal{D}(t)  \dt.
\int_\Omega 1\,|\textrm{d}\,r^M| \lesssim  \mathcal{D}, \quad \mbox{for a.e. in }  (0,T).
\end{equation}

%Similarly, the energy inequality (\ref{energy_stability}) can be written as
%\begin{equation} \label{M7}
%\begin{split}
%\left[ \intO{ \frac{1}{2} \left< \Nu_{t,x}; \vr | {\bu } |^2 + \Hc(\vr) \right>  } \right]_{t = 0}^{t = \tau} +
%\int_0^\tau \intO{  \left( \mu |\Grad \bu |^2 + (\mu + \lambda) |\Div \bu |^2 \right) } \dt  + \mathcal{D}(\tau) &\leq 0
%\end{split}
%\end{equation}
%for a.e. $\tau \in [0,T]$, with the \emph{dissipation defect} $\mathcal{D}$ satisfying
%\begin{equation} \label{M8}
%\int_0^\tau \| r^M \|_{\mathcal{M}(\Omega)}\ \dt \le  \int_0^\tau \mathcal{D}(t) \ \dt,\
%\mathcal{D}(\tau) \geq  \liminf_{h \to 0} \int_0^\tau  \norm{\Gradedge\vuh}_{L^2}^2  \dt - \int_0^\tau \intO{ |\nabla _x \bu |^2 }\ \dt,
%\end{equation}
%cf. \cite[Lemma 2.1]{FGSWW}.

Collecting \eqref{M4}--\eqref{M9} implies that the Young measure $\{ \Nu_{t,x} \}_{t,x \in (0,T) \times \Omega}$ represents
a dissipative measure-valued solution of the Navier--Stokes system (\ref{ns_eqs})
in the sense of Definition \ref{def_dmvs}.
Seeing that validity of (\ref{M4}) and (\ref{M6}) can be extended to the class of test functions from $C^1([0,T] \times  \Omega ; R^d)$, we have proved Theorem \ref{thm_dmvs}.
\end{proof}

\subsection{Convergence to strong solution}
In the previous  subsection, we have shown that the numerical solution generates the dissipative measure-valued solution.
We admit that the conclusion of Theorem~\ref{thm_dmvs} is rather weak, also due to the non-uniqueness of Young measure. However,
we may directly use the DMV-strong uniqueness principle established in \cite[Theorem 4.1]{FGSWW} to obtain convergence to the strong solution as long as it exists.
\begin{Theorem}[Convergence to strong solution]\label{thm_convergence}
In addition  to the hypotheses of Theorem \ref{thm_dmvs}, suppose that the Navier--Stokes system (\ref{ns_eqs}) endowed with the initial data $(\vr_0, \bfu_0)$  admits a strong solution $(\vr,\bfu)$  belonging to the class
\begin{equation*}
\vr, \Grad \vr, \bfu, \Grad \bfu \in C([0,T]\times {\Omega}),\  \partial_t \bfu \in L^2\left(0,T;C( {\Omega};R^d) \right),\ \vr>0.
\end{equation*}
Then
\begin{equation*}
\vrh \rightarrow \vr \mbox{ (strongly) in } L^{\gamma}\left((0,T)\times \Omega \right),\
\bfu_h \rightarrow \bfu \mbox{ (strongly) in } L^2\left((0,T)\times \Omega; R^d \right).
\end{equation*}
%\textcolor{blue}{Ed:$\Omega$ is the torus, in particular, it is compact.}
\end{Theorem}
\begin{Remark}
By strong solution we mean that all generalized derivatives appearing in the equations can be identified with integrable functions, see \cite{FGSWW}.
\end{Remark}

Indeed, the DMV--strong uniqueness implies that the Young measure generated by the family of numerical solutions coincides at
a.a. point $(t,x)$ with the Dirac mass supported by the smooth solution of the problem. In particular, the numerical solutions converge strongly and no oscillations occur.

\medskip

\begin{Remark}

We have constructed solution on a space-periodic domain $\Omega$. When considering a polyhedral domain, the existence of smooth solutions remains open and may be a delicate task. To avoid this problem, one has to approximate a smooth domain by a family of polyhedral domains analogously as in \cite{FeLu18_CR}. Note, however, this problem does not occur in the case of periodic domain.
\end{Remark}

If, in addition, we assume the density is uniformly bounded, meaning independently of the numerical step, the results of Theorems~\ref{thm_dmvs} and \ref{thm_convergence} may be possibly extended to an  unstructured grid. Indeed the only difference of the proof would be  showing the consistency of the convective terms in \eqref{consistency_convective}. The estimate of the error terms $E_1(\vrh)$ and $E_1(\vrh\vuh)$ could be done without the discrete integration by parts thanks to $L^\infty-$bound on the density.
%Another way would be to introduce new discrete operators $\pdedgei r_h,$ $\pdmeshi q_{i,h}$ between the dual and the unstructured primary grid, %such that the discrete integration by parts holds.
%% ML
%% They can be defined for the unstructured grid, as shown in the literature, see my overview part
%% precisely, since the discrete operators $\pdedgei r_h,$ $\pdmeshi q_{i,h}$ and $\Delta_h^{(i)} r_h$ can not be defined on an unstructured grid,
Moreover, in view of the conditional regularity result \cite{Sun}, we obtain the
unconditional convergence to the strong solution since the DMV solution with bounded density is regular.

\begin{Theorem}[Convergence with bounded density]\label{thm_convergence_BD}
Let $d = 3$.
In addition  to the hypotheses of Theorem \ref{thm_dmvs}, suppose that
\begin{itemize}
\item
the initial data belong to the class
\[
\vr_0 \in W^{3,2}(\Omega),\ \vc{u}_0 \in W^{3,2}(\Omega; R^d);
\]
\item
bulk viscosity vanishes, meaning
 \[
 \lambda + \frac{2}{3} \mu = 0;
 \]
 \item
 \[
 \| \varrho_h \|_{L^\infty((0,T) \times \Omega)} \leq c
 \]
 uniformly for $h \to 0$.

\end{itemize}

Then
\begin{equation*}
\vrh \rightarrow \vr \mbox{ (strongly) in } L^{q}\left((0,T)\times \Omega \right),\ q \geq 1,\
\bfu_h \rightarrow \bfu \mbox{ (strongly) in } L^2\left((0,T)\times \Omega; R^d \right),
\end{equation*}
$(\varrho, \bfu)$ is the strong solution to the Navier--Stokes system (\ref{ns_eqs}) with the initial data $(\vr_0, \bfu_0)$.
\end{Theorem}

The condition on vanishing bulk viscosity is technical and we refer to \cite{Sun} for the discussion of its necessity. We point out that Theorem
\ref{thm_convergence_BD} guarantees \emph{unconditional} convergence of the scheme without the {\it a priori} hypothesis of the existence of smooth solution. In other words, uniform boundedness of the numerical densities implies  the existence of global smooth solution as long as the initial data are sufficiently regular.
It is also worth noting that boundedness of the numerical densities is still a considerably weaker assumptions than the hypothesis made by Jovanovi\' c \cite {jovanovic}.

\section{Numerical experiment}\label{sec_numerics}
In this section we show the numerical performance of scheme \eqref{scheme} in  two space dimensions.
Note that scheme \eqref{scheme} is nonlinear,  thus we solve it numerically by a fixed-point iteration. For each sub-iteration, we set the time step as $\TS = \text{CFL} \frac{h}{(|u| + c)_{\max}}$, where
$\text{CFL}=0.3$, $c=\sqrt{\gamma p/\rho}$.  We set the viscosity coefficients $\mu=\lambda=0.01$ and the adiabatic coefficient $\gamma=1.4$ in all experiments. Moreover, we choose the artificial diffusion $\eps=0.6$ which satisfies the assumption of $0< \eps <\min\{1, 2(\gamma-1)\}$.\\
\bigskip

{\bf Experiment 1.}
 First we validate the accuracy of the scheme by considering
\[ \rho_{ref} = 2+ \cos (2 \pi (x+y)), \quad \vu_{ref}= \left( \frac{\sin (2\pi t)}{2+\cos(2\pi(x+y))}, -\frac{\sin (2\pi t)}{2+\cos(2\pi(x+y))} \right)^T
\]
with the corresponding driving force in the momentum equation.

We compute the relative error $e_{\phi_h}$ for $\phi \in \{\vr, \vu, \nabla \vu \}$ in the corresponding norms, and the experimental order of convergence (EOC),  where
\[e_{\phi_h}=\frac{\norm{\phi_h - \phi_{ref}}}{\norm{\phi_{ref}}}, \quad \text{EOC}= \log_2\frac{e_{\phi_{2h}}}{e_{\phi_h}},
\]
and $\phi_{ref}$ denotes the reference solution. From the numerical results, we observe the first order of convergence of the scheme, see Table~\ref{tab:smooth}.\\
\begin{table}[h!]\centering
\caption{Numerical convergence for Experiment~1.}
\label{tab:smooth}
\begin{tabular}{c|c|c|c|c|c|c|c|c}\hline
  h  &  $\norm{e_{\Grad \vu}}_{L^2(L^2)}$ &  EOC  &  $\norm{e_{\vu}}_{L^2(L^2)}$ &  EOC    &  $\norm{e_{\vr}}_{L^1(L^1)}$ &  EOC  & $\norm{e_{\vr}}_{L^{\infty}(L^{\gamma})}$ &  EOC  \\ \hline
  1/32 & 4.21e-02	& --    &3.43e-03	 & --   &1.24e-03  & --   &4.28e-02  & -- \\
  1/64 & 1.78e-02 	&1.24	&1.39e-03	 &1.30	&4.95e-04  &1.32  &1.81e-02  &1.24  \\
  1/128 &7.75e-03	&1.20	&5.88e-04	 &1.24 	&2.04e-04  &1.28  &7.86e-03  &1.21  \\
  1/256 & 3.51e-03	&1.14 	&2.59e-04	 &1.18 	&8.69e-05  &1.23  &3.50e-03  &1.17 \\
\hline
\end{tabular}
\end{table}

{\bf Experiment 2.}
In this experiment, we simulate the Gresho--vortex flow~\cite{DeTa, NBALM, hoseksheMAC}. The initial state is the vortex of radius $r_0=0.2$ located at $(0.5,0.5)$ with
\[
\vr(0,\bfx) =1, \quad
\vu(0,\bfx)= \left( \begin{array}{c}
y-0.5 \\ 0.5-x
\end{array}\right)\,\frac{u_r(r)}{r}, \quad  \text{with} \quad
u_r(r)= \sqrt{\gamma}\left\{ \begin{array}{cc}
2 r/r_0 & \mbox{if } 0\leq  r<r_0/2, \\
2 (1- r/r_0)  & \mbox{if } r_0/2\leq  r< r_0, \\
0 & \mbox{if } r\geq r_0 ,
\end{array}\right.
\]
where $r=\sqrt{(x-0.5)^2 +(y-0.5)^2}$.
We present the evolution of the flow in Figure~\ref{fig:gresho_vortex}  for the mesh size $h=1/128$.  We can clearly recognize that the solution is in a good agreement with those presented in the literature, see \cite{hoseksheMAC}.
To further invest the numerical convergence, we present in Table \ref{tab:gresho_vortex} the errors for different mesh parameters and the reference solution is computed at a fine mesh $h=1/2048$. We observe even better than the first order convergence.
%%%%%%%%%%%%%%%%%%%%%%%%%%%%%%%%%%%%%%%%%%%%%%%%%%%%%%%%

\begin{figure}[h!]
\centering
%%%%%%%%%%%%%%%%%%%%%%%%%%%%%%%%%
\begin{subfigure}{0.32\textwidth}
\includegraphics[width=\textwidth]{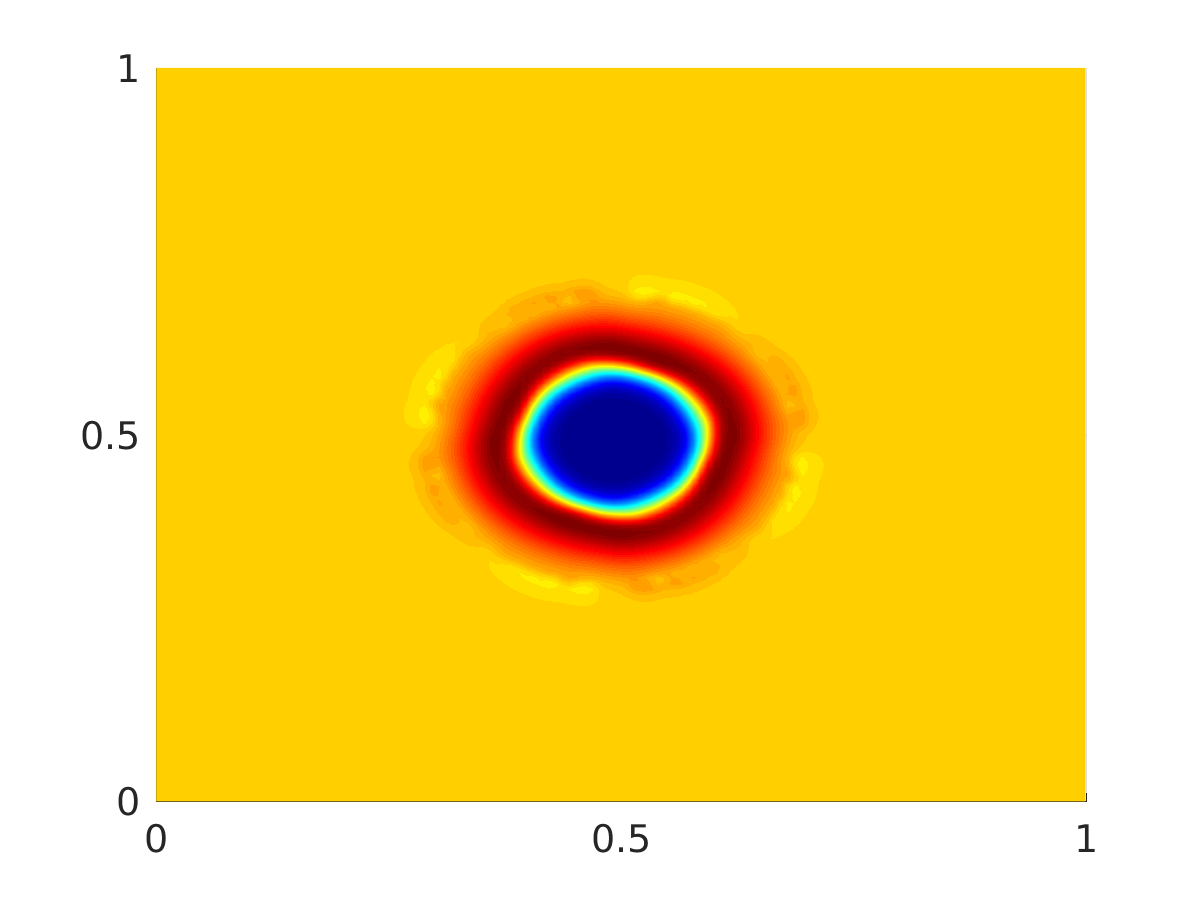}
\end{subfigure}
\begin{subfigure}{0.32\textwidth}
\includegraphics[width=\textwidth]{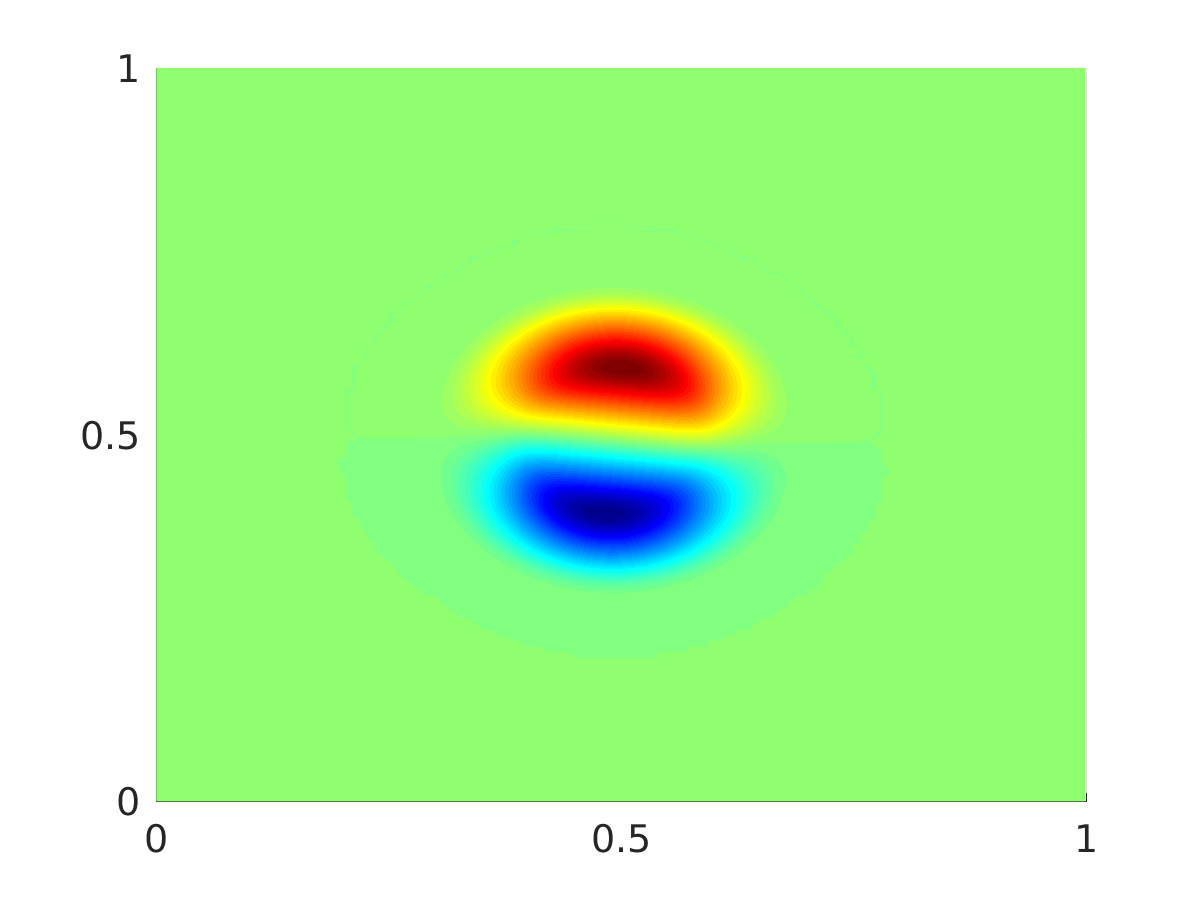}
\end{subfigure}
\begin{subfigure}{0.32\textwidth}
\includegraphics[width=\textwidth]{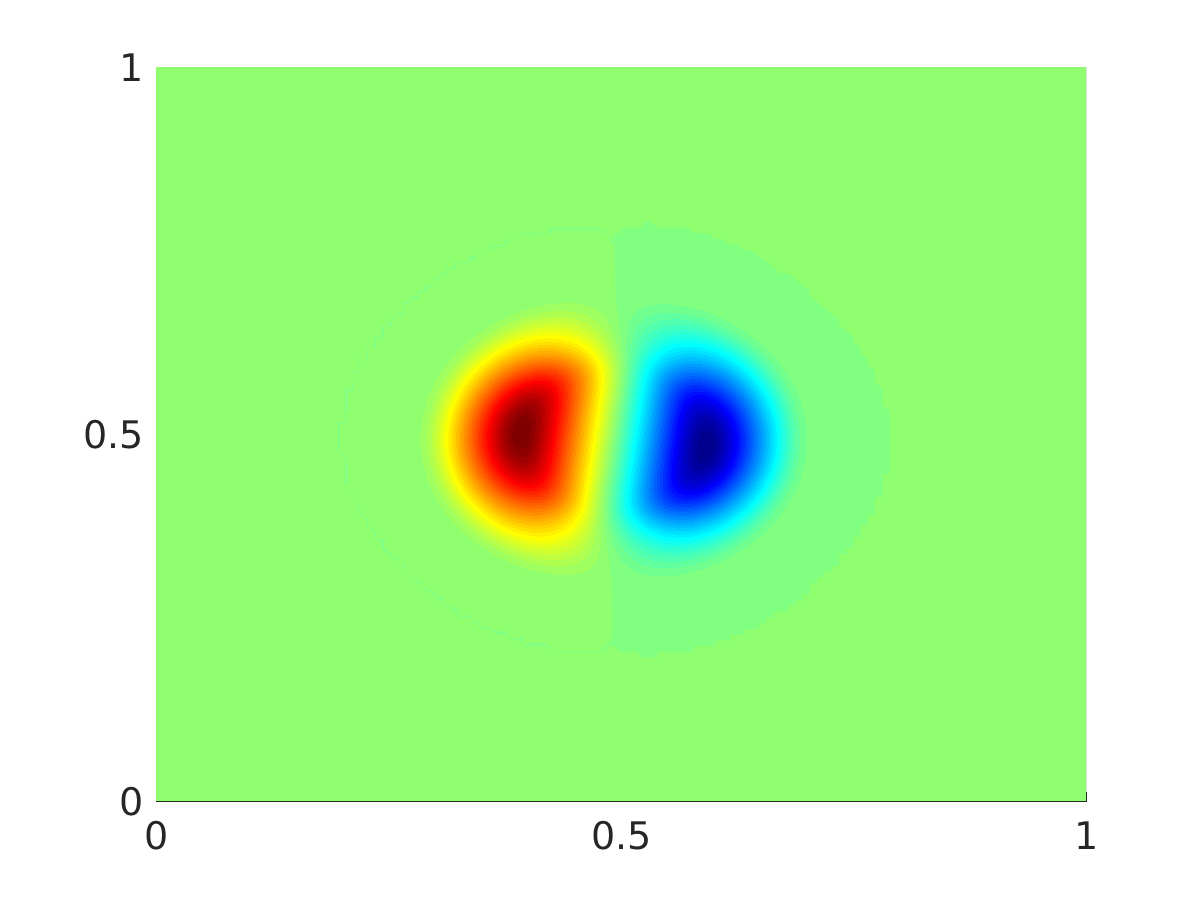}
\end{subfigure}\\
%%%%%%%%%%%%%%%%%%%%%%%%%%%%%%%%%
\begin{subfigure}{0.32\textwidth}
\includegraphics[width=\textwidth]{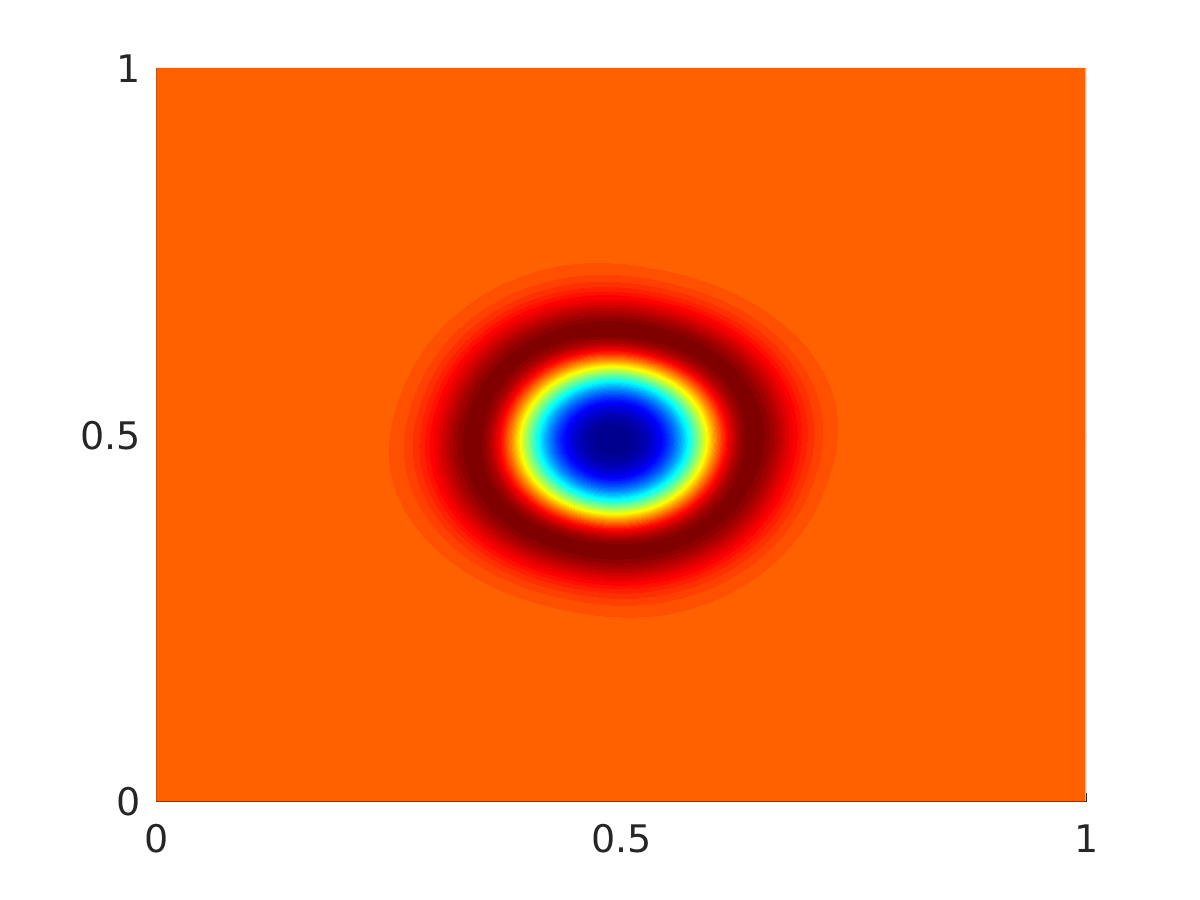}
\end{subfigure}
\begin{subfigure}{0.32\textwidth}
\includegraphics[width=\textwidth]{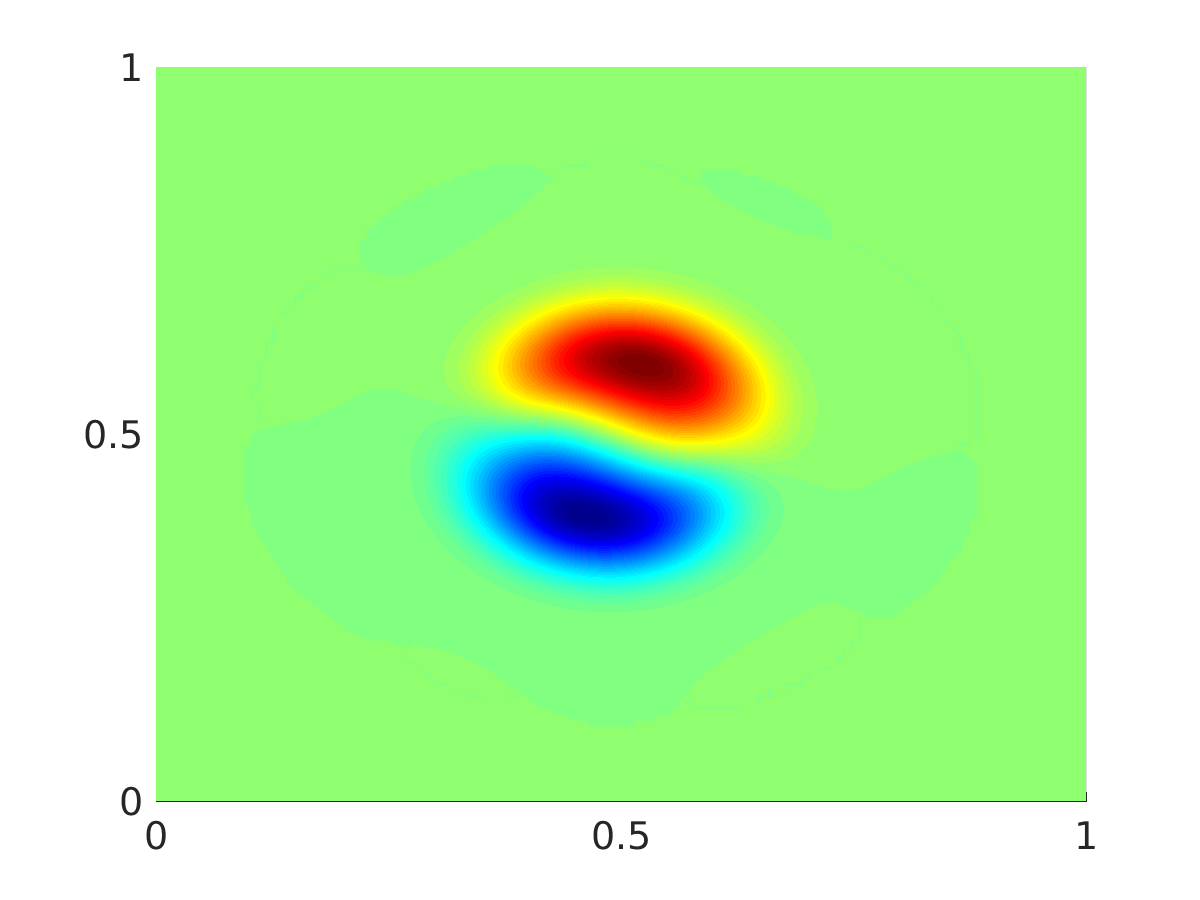}
\end{subfigure}
\begin{subfigure}{0.32\textwidth}
\includegraphics[width=\textwidth]{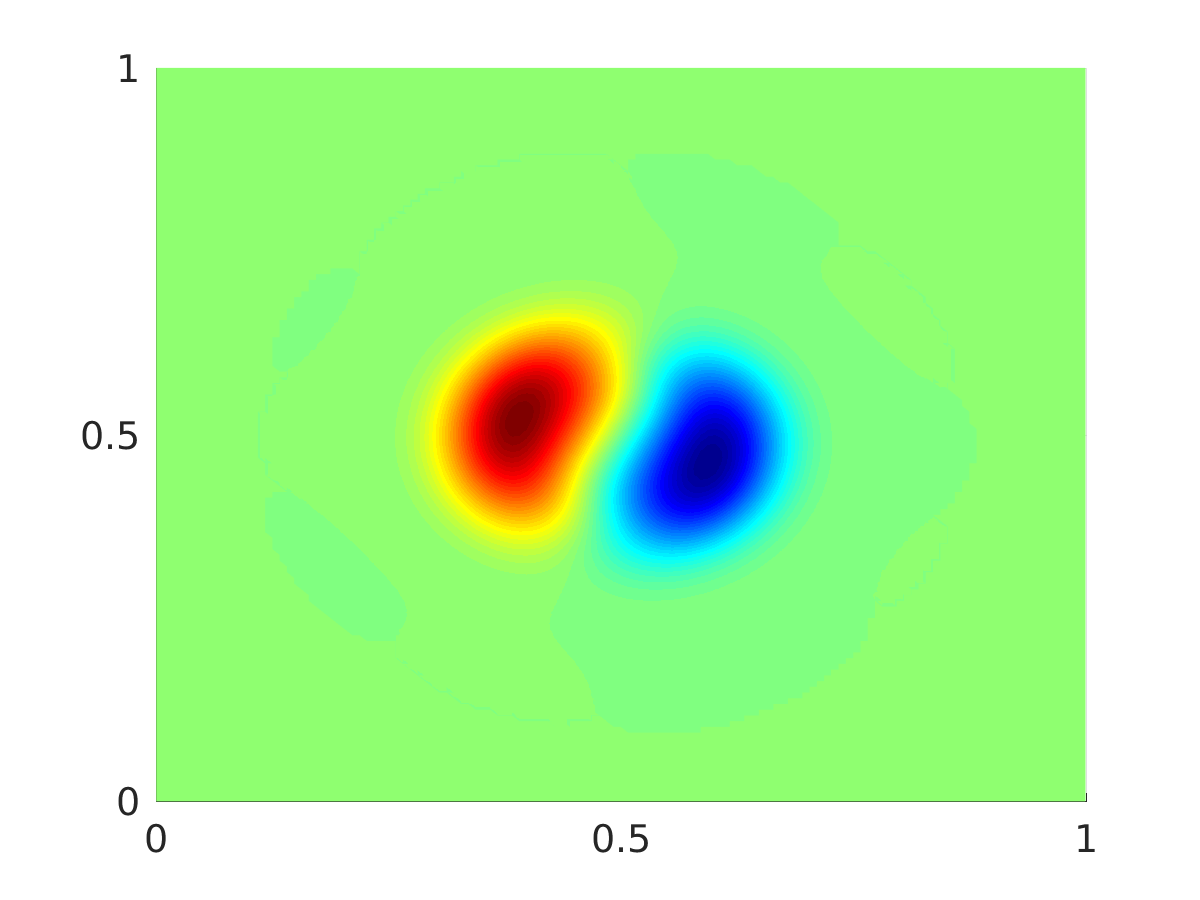}
\end{subfigure}\\
%%%%%%%%%%%%%%%%%%%%%%%%%%%%%%%%%
\begin{subfigure}{0.32\textwidth}
\includegraphics[width=\textwidth]{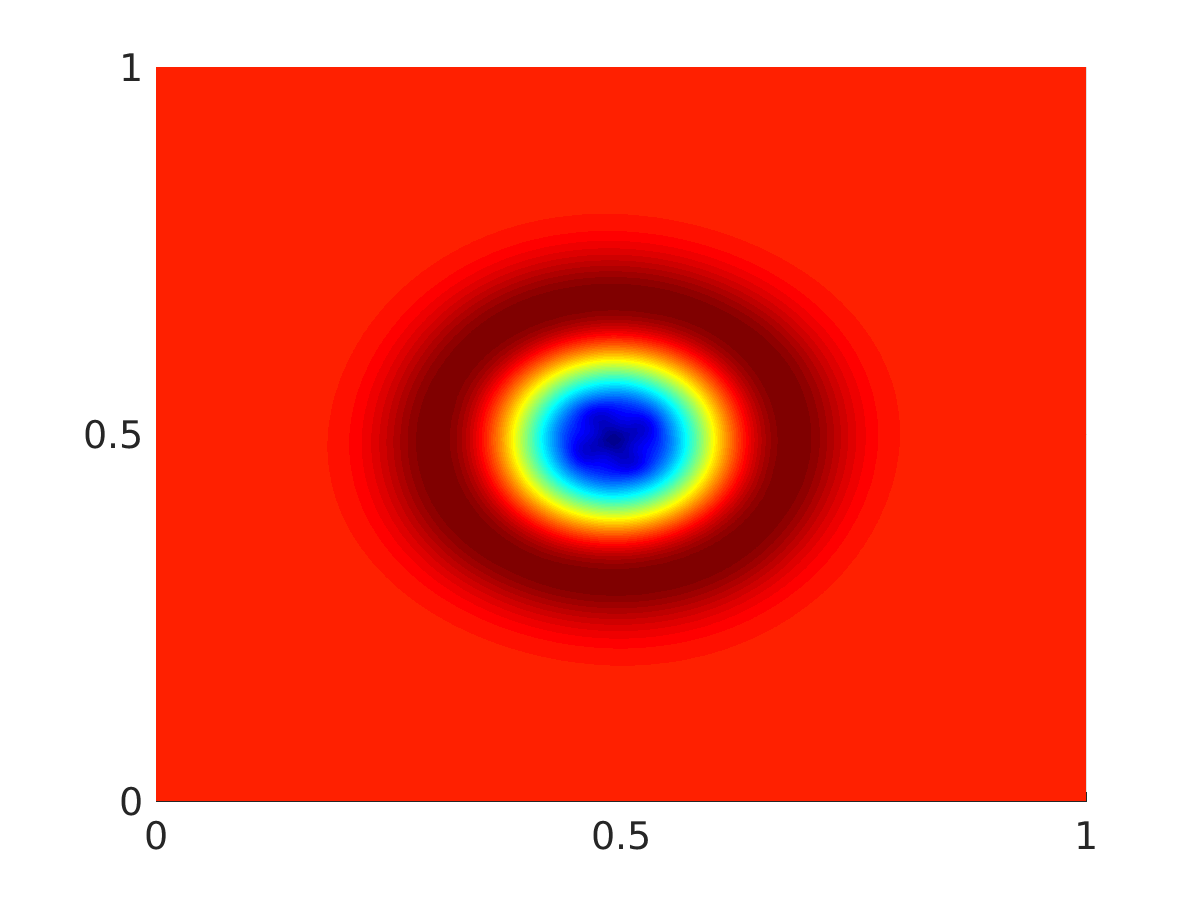}
\end{subfigure}
\begin{subfigure}{0.32\textwidth}
\includegraphics[width=\textwidth]{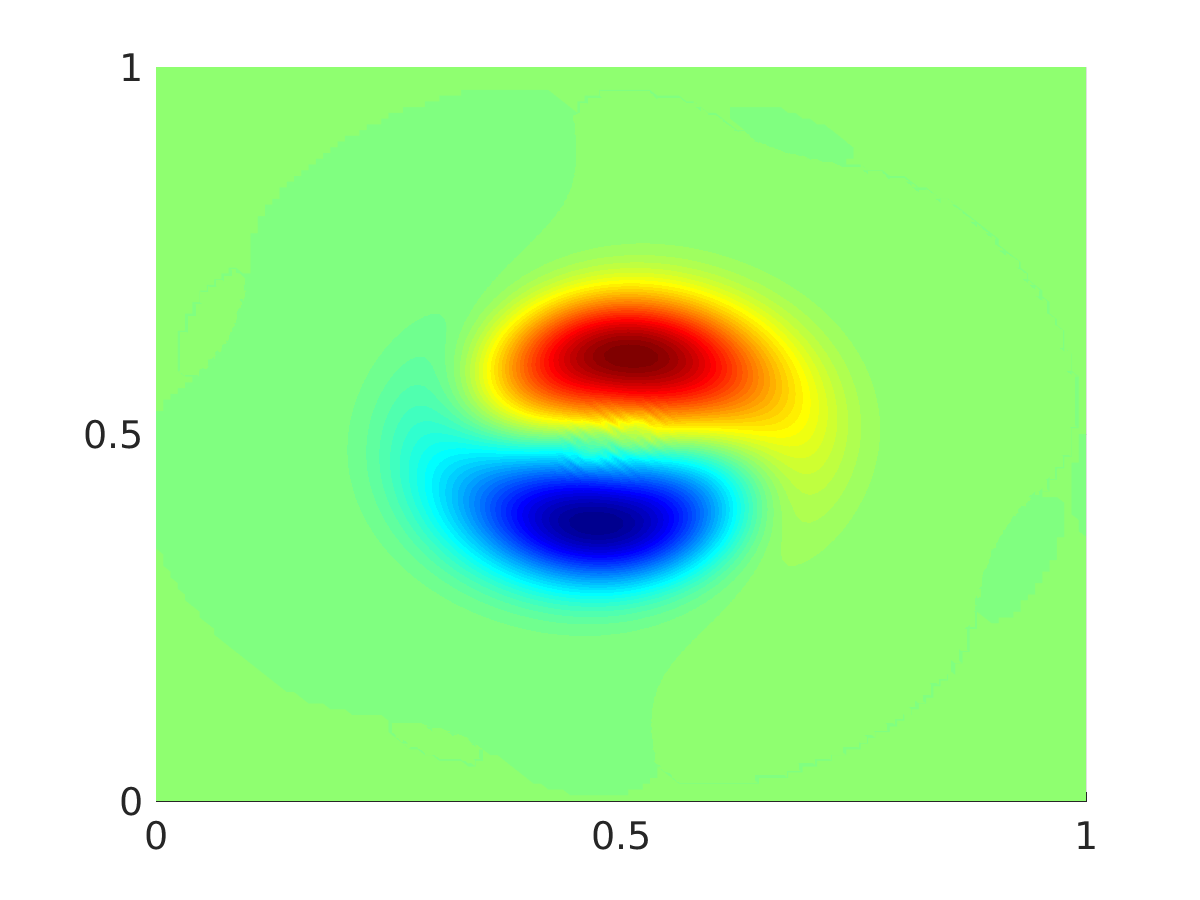}
\end{subfigure}
\begin{subfigure}{0.32\textwidth}
\includegraphics[width=\textwidth]{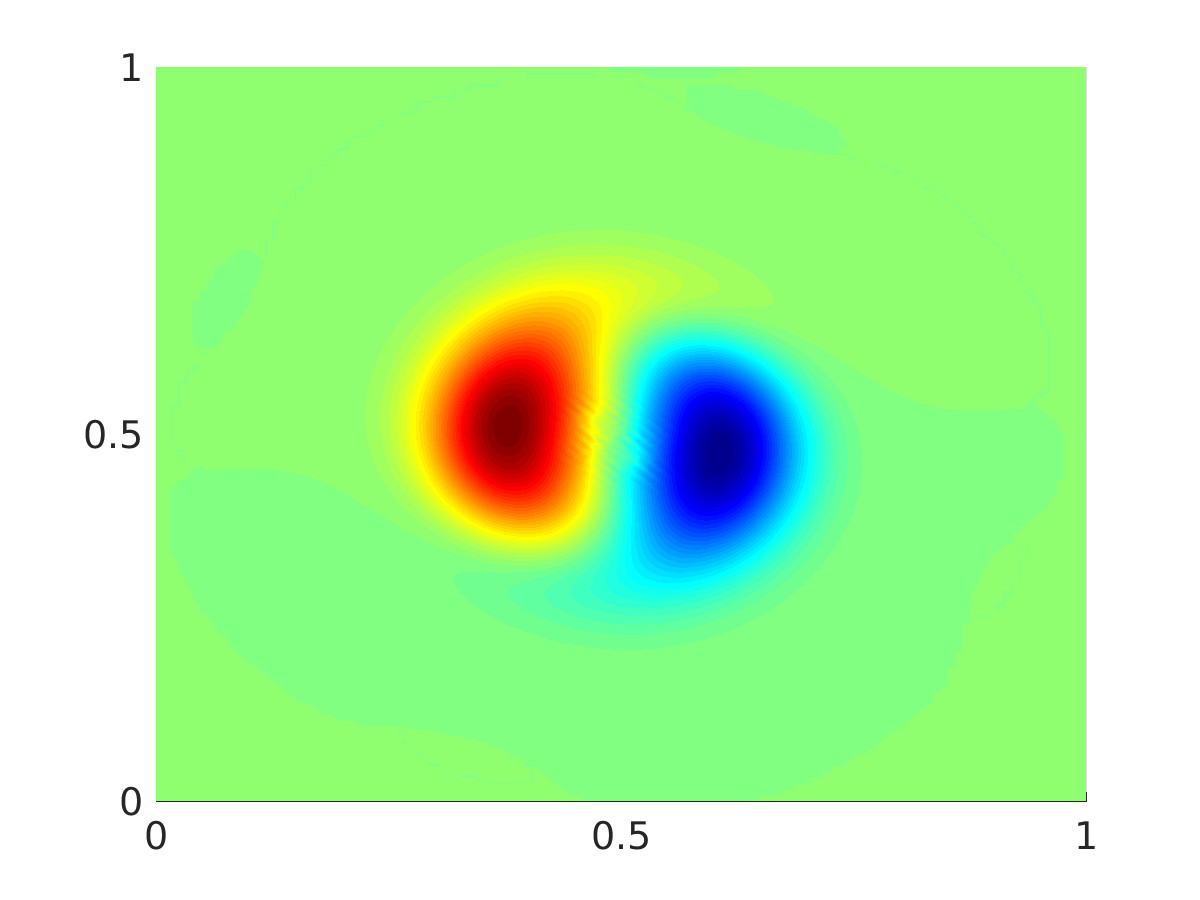}
\end{subfigure}\\
%%%%%%%%%%%%%%%%%%%%%%%%%%%%%%
%%%%%%%%%%%%%%%%%%%%%%%%%%%%%%%%%
\begin{subfigure}{0.32\textwidth}
\includegraphics[width=\textwidth]{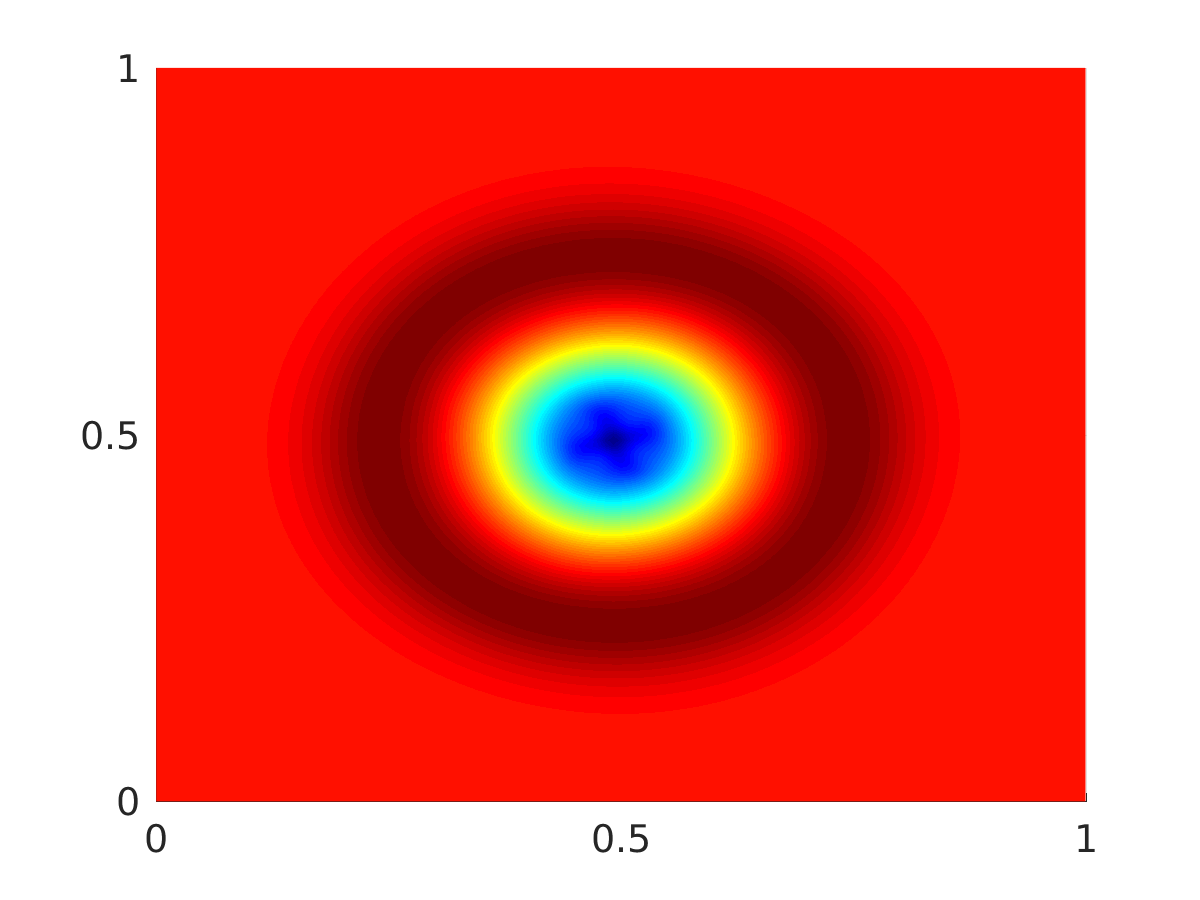}
\end{subfigure}
\begin{subfigure}{0.32\textwidth}
\includegraphics[width=\textwidth]{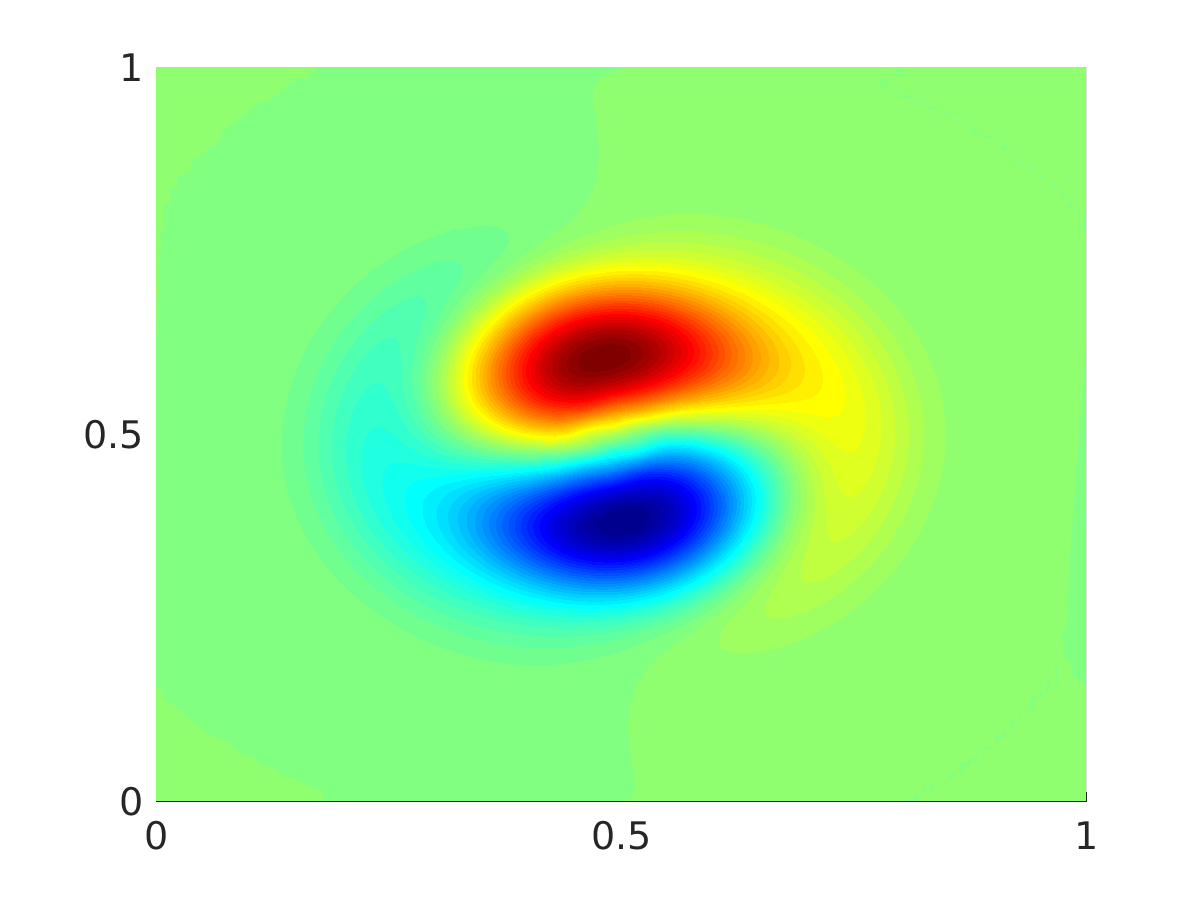}
\end{subfigure}
\begin{subfigure}{0.32\textwidth}
\includegraphics[width=\textwidth]{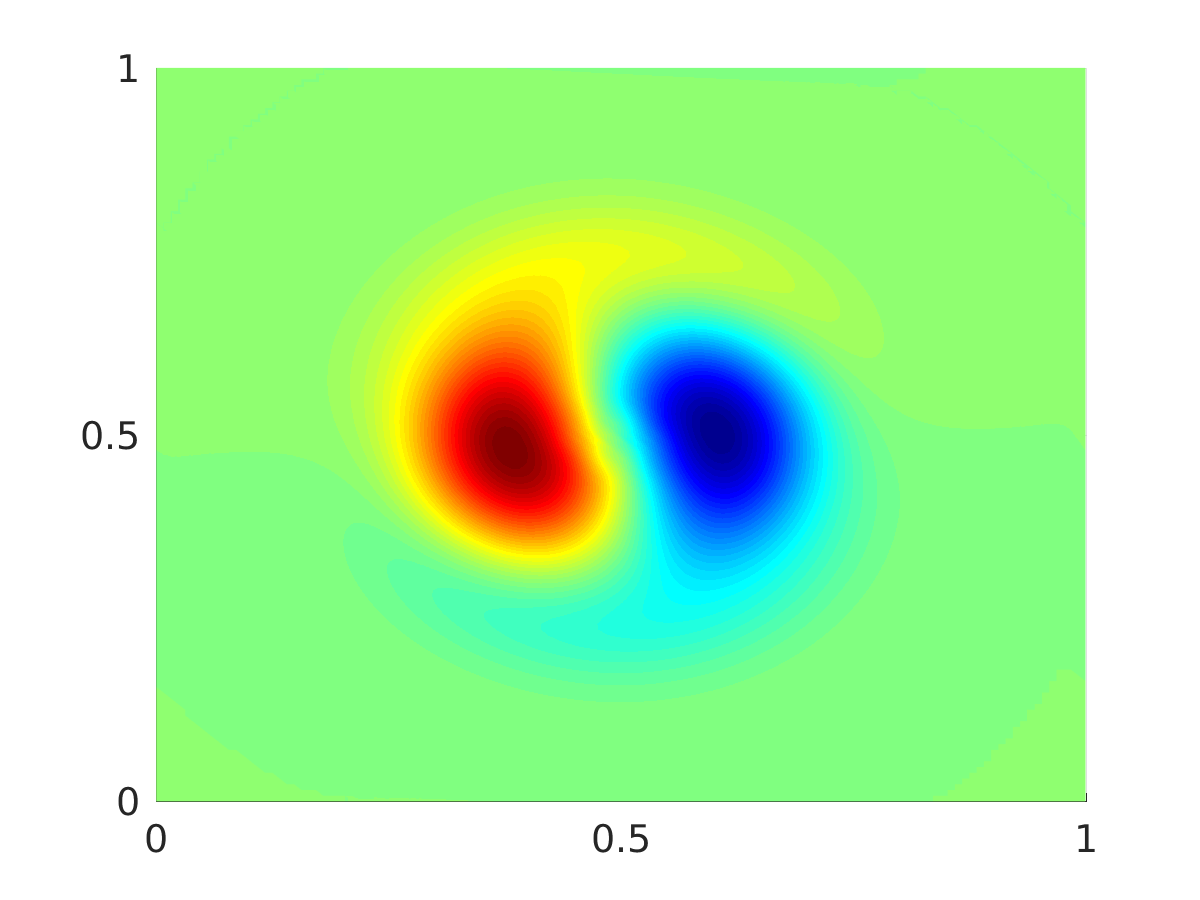}
\end{subfigure}\\
%%%%%%%%%%%%%%%%%%%%%%%%%%%%%%
\begin{subfigure}{0.32\textwidth}
\includegraphics[width=\textwidth]{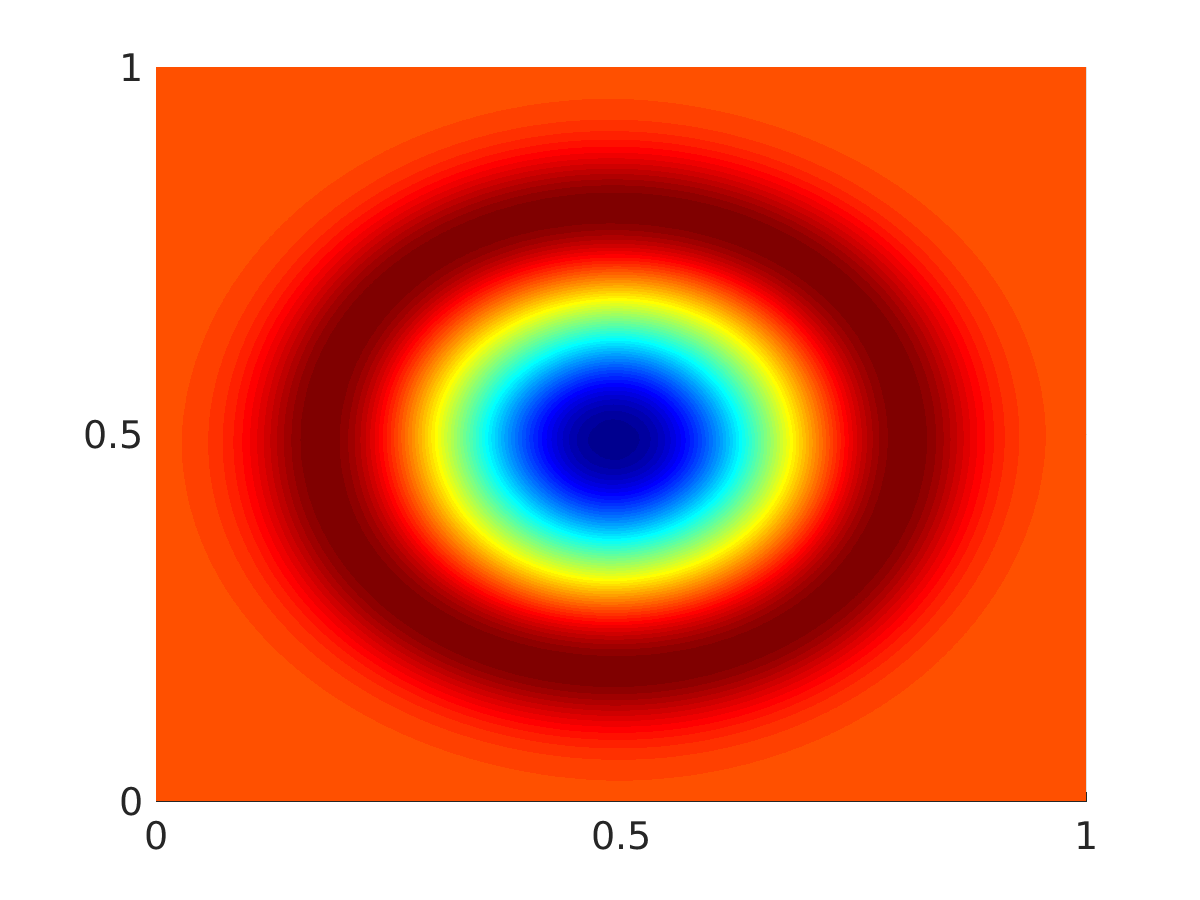}
\caption{density $\varrho$}
\end{subfigure}
\begin{subfigure}{0.32\textwidth}
\includegraphics[width=\textwidth]{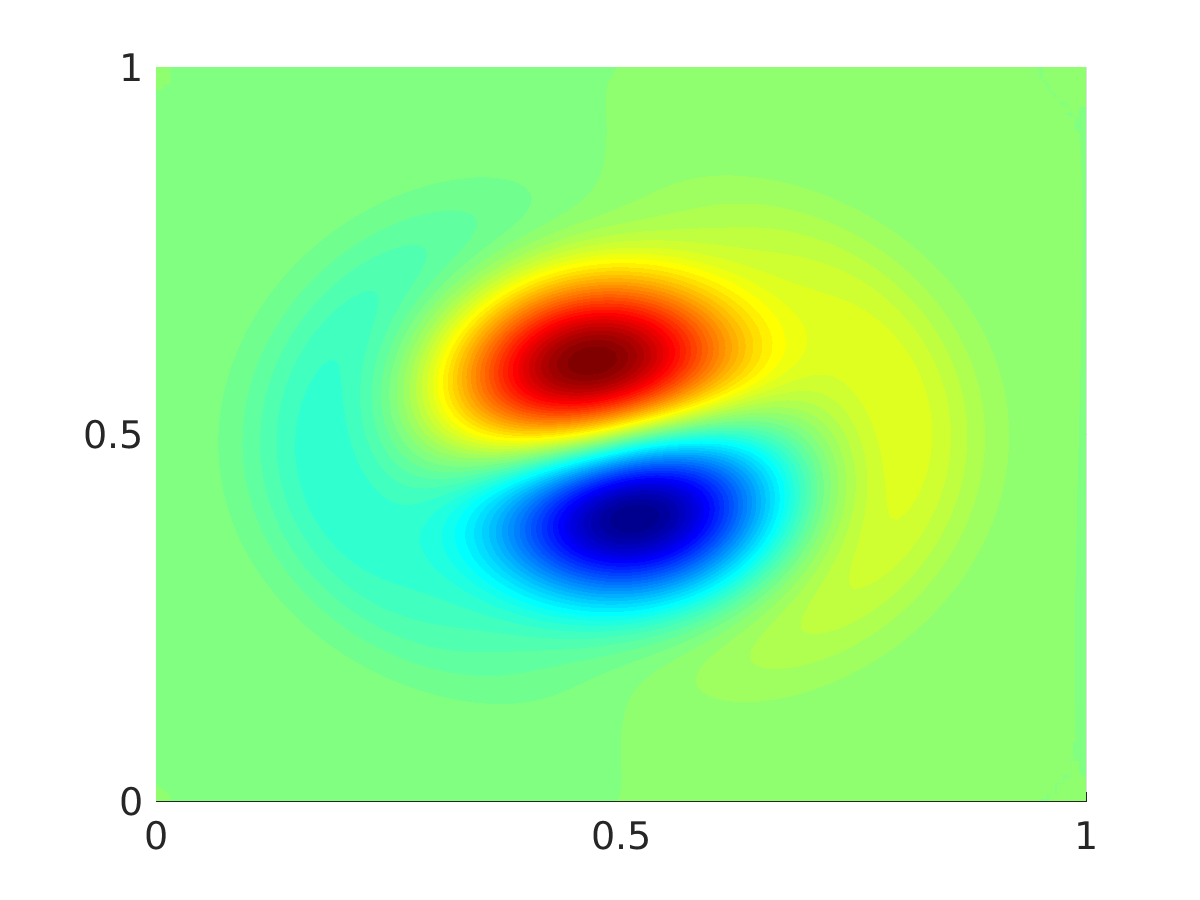}
\caption{velocity component $u_1$}
\end{subfigure}
\begin{subfigure}{0.32\textwidth}
\includegraphics[width=\textwidth]{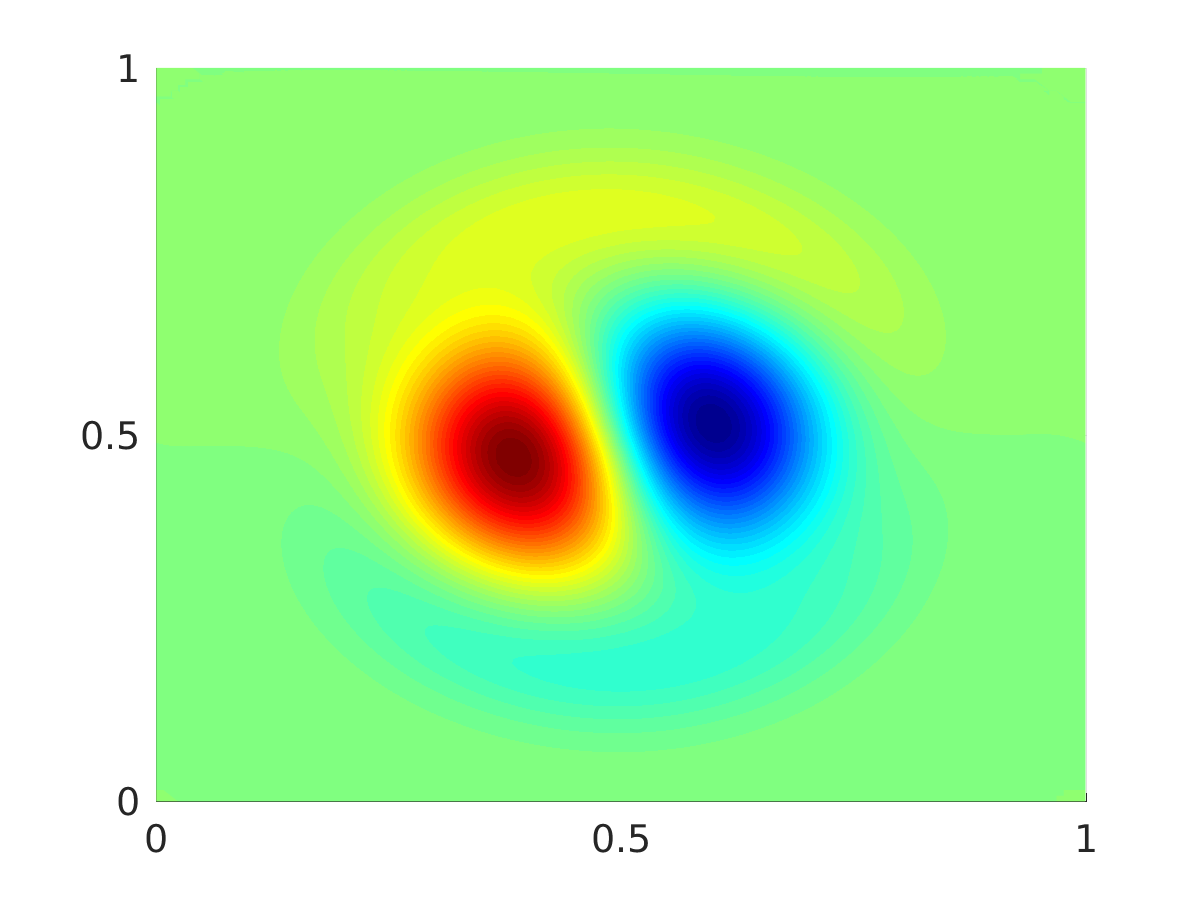}
\caption{velocity component $u_2$}
\end{subfigure}
%%%%%%%%%%%%%%%%%%%%%%%%%%%%%%%%
\caption{Time evolution of  the Gresho--vortex: solution at $t=0.01, 0.05, 0.1, 0.15, 0.2$ from top to bottom, solution of density and velocity components from left to right}
\label{fig:gresho_vortex}
\end{figure}

\begin{table}[h!]\centering
\caption{Numerical convergence for Experiment~2.}
\label{tab:gresho_vortex}
\begin{tabular}{c|c|c|c|c|c|c|c|c}\hline
  h  &  $\norm{e_{\Grad \vu}}_{L^2(L^2)}$ &  EOC  &  $\norm{e_{\vu}}_{L^2(L^2)}$ &  EOC    &  $\norm{e_{\vr}}_{L^1(L^1)}$ &  EOC  & $\norm{e_{\vr}}_{L^{\infty}(L^{\gamma})}$ &  EOC  \\ \hline
% 1/32 & 7.98e-01	& --    & 3.68e-02	 & --   & 3.90e-04 & --   & 1.03e-02 & -- \\
%  1/64 & 4.42e-01 	& 0.85	&1.89e-02	 & 0.96	& 2.16e-04 & 0.85 & 5.81e-03 & 0.83 \\
%  1/128 & 2.17e-01	&1.03	&8.95e-03 	 & 1.07	& 1.06e-04 & 1.03 & 2.83e-03 & 1.04 \\
%  1/256 & 9.43e-02	& 1.20	&3.86e-03	 & 1.21	& 4.72e-05 & 1.17 & 1.25e-03 & 1.18\\
1/32 & 6.66e-01 & --  & 3.16e-02 & --  & 6.64e-04 & --  & 1.64e-02 & --  \\
1/64 & 3.75e-01 & 0.83   & 1.66e-02 & 0.93   & 3.60e-04 & 0.88   & 8.85e-03 & 0.89   \\
1/128 & 1.91e-01 & 0.97   & 8.21e-03 & 1.01   & 1.80e-04 & 1.00   & 4.43e-03 & 1.00   \\
1/256 & 9.11e-02 & 1.07   & 3.86e-03 & 1.09   & 8.51e-05 & 1.08   & 2.09e-03 & 1.08   \\
 1/512 & 3.93e-02 & 1.21   & 1.66e-03 & 1.22   & 3.66e-05 & 1.22   & 8.96e-04 & 1.22  \\
 1/1024 & 1.31e-02 & 1.58   & 5.54e-04 & 1.58   & 1.21e-05 & 1.60   & 2.96e-04 & 1.60   \\

\hline
\end{tabular}
\end{table}

\section*{Conclusion}
We have studied a finite volume method for the multi--dimensional compressible isentropic Navier--Stokes equations on regular quadrilateral mesh in a periodic domain.
Due to  the artificial diffusion in the numerical flux function \eqref{num_flux} we have sufficiently strong a priori estimate on jumps of the discrete density. The solutions of the scheme were shown to exist while preserving the positivity of the  discrete density.  Moreover, we have shown the stability of the scheme by deriving the unconditional balance of the discrete total energy in Theorem~\ref{theorem_stability}. Furthermore, we have established the consistency formulation  provided the artificial diffusion coefficient is large enough, see Theorem~\ref{Tm2}.
In addition, we have shown in Theorem~\ref{thm_dmvs} that the numerical solutions of scheme \eqref{scheme} generate a DMV solution of the Navier--Stokes system~\eqref{ns_eqs}.
Finally, using the recent result on the DMV--strong uniqueness principle and the conditional regularity result \cite{Sun}, we have proven the convergence to the strong solution assuming only the existence of the latter, cf.
%on its lifespan
Theorem~\ref{thm_convergence}, and the unconditional convergence to  regular solution, cf. Theorem~\ref{thm_convergence_BD}.
%the convergence to the strong solution on its lifespan and unconditional convergence to regular solution, cf. Theorem~\ref{thm_convergence} and Theorem~\ref{thm_convergence_BD}, respectively.
Numerical experiments are also presented to support the theoretical results.
To the best of our knowledge, this is the first rigorous result concerning convergence of a finite volume method for the compressible isentropic Navier--Stokes equations in the multi--dimensional setting.

\def\cprime{$'$} \def\ocirc#1{\ifmmode\setbox0=\hbox{$#1$}\dimen0=\ht0
  \advance\dimen0 by1pt\rlap{\hbox to\wd0{\hss\raise\dimen0
  \hbox{\hskip.2em$\scriptscriptstyle\circ$}\hss}}#1\else {\accent"17 #1}\fi}

\bibliography{citace}

\begin{thebibliography}{10}
\bibitem{ABLV}
G.~Ansanay-Alex, F.~Babik, J.~C.~Latch\'e, and D.~Vola.
\newblock An L2--stable approximation of the Navier--Stokes convection operator for low--order non-conforming finite elements
\newblock{\em Int. J. Numer. Meth. Fluids}, {\bf 66}: 555--580, 2011.

\bibitem{BALL2}
J.M. Ball.
\newblock A version of the fundamental theorem for {Y}oung measures.
\newblock {\em In Lect. Notes in Physics 344, Springer}, 207--215,  1989.

\bibitem{B2012}
P.~Birken.
\newblock  {\em Numerical Methods for the
Unsteady Compressible Navier--Stokes Equations.}
\newblock Habilitation Thesis, Kassel, 2012.

\bibitem{ciarlet}
P.~G.~Ciarlet.
\newblock The Finite Element Method for Elliptic Problems.
\newblock{\em Classics in Applied Mathematics, Society for Industrial and Applied Mathematics}, 2002.


\bibitem{DeTa}
P. Degond,  and M.~Tang.
\newblock All speed scheme for the low Mach number limit of the isentropic Euler equations.
\newblock{\em Commun. Comput. Phys. }{\bf 10}: 1--31,  2011.

\bibitem{DoFei}
V.~Dolej\v{s}\'i and M.~Feistauer.
\newblock Discontinuous Galerkin method.
\newblock{\em Vol.~48  Springer Series in Computational Mathematics, Springer},  2015.



\bibitem{feist1}
M.~Feistauer.
\newblock {\em Mathematical Methods in Fluid Dynamics}.  Pitman Monographs and Surveys in Pure and Applied Mathematics, Vol.~67.
Longman Scientific \& Technical, Harlow, 1993.

\bibitem{FFL1995}
M.~Feistauer, J.~Felcman, and M.~Luk\'a\v{c}ov\'a-Medvid'ov\'a.
\newblock Combined finite element-finite volume solution of compressible flow.
\newblock {\em J. Comput. Appl. Math.} {\bf 63}: 179--199, 1995.

\bibitem{G2018}
G.J.~Gassner, A.R.~Winters, F.J.~Hindenlang, and D.A.~Kopriva.
\newblock The BR1 Scheme is stable for the compressible Navier-Stokes equations.
\newblock {\em J. Sci. Comput.} {\bf 77}(1): 154--200, 2018.


\bibitem{EyGaHe}
R.~Eymard, T.~Gallou{\"{e}}t, and R.~Herbin.
\newblock Finite volume methods.
\newblock{\em Handbook of numerical analysis} {\bf 7}: 713--1018, 2000.



\bibitem{FGSWW}
E.~Feireisl, P.~Gwiazda, A.~{\'S}wierczewska-Gwiazda, and E.~Wiedemann.
\newblock Dissipative measure-valued solutions to the compressible
  {N}avier--{S}tokes system.
\newblock {\em Calc. Var. Partial Dif.} {\bf 55}(6): 55--141, 2016.

\bibitem{FeKaNo}
E.~Feireisl, T.~Karper, and A.~Novotn{\'y}.
\newblock A convergent numerical method for the {N}avier--{S}tokes--{F}ourier   system.
\newblock {\em IMA J. Numer. Anal.} {\bf 36}(4): 1477--1535, 2016.



\bibitem{FeLu18_CR}
E.~Feireisl and M.~Luk\'a\v{c}ov\'a-Medvid'ov\'a.
\newblock Convergence of a mixed finite element--finite volume  scheme for the isentropic Navier--Stokes system via the dissipative measure--valued solutions.
\newblock{\em Found. Comput. Math.} {\bf 18}(3): 703--730, 2018.



\bibitem{FLM18_brenner}
E.~Feireisl, M.~Luk\'a\v{c}ov\'a-Medvid'ov\'a, and H.~Mizerov\'a.
\newblock A finite volume scheme for the Euler system inspired by the two velocities approach. ArXiv: \url{arxiv.org/abs/1805.05072},  2018.

\bibitem{FLM18_euler}
E.~Feireisl, M.~Luk\'a\v{c}ov\'a-Medvid'ov\'a, and H.~Mizerov\'a.
\newblock Convergence of finite volume schemes for the Euler equations via dissipative measure--valued solutions.
ArXiv: \url{arxiv.org/abs/1803.08401},  2018.

\bibitem{FeiNov_book}
E.~Feireisl, and A.~Novotn{\'y}.
\newblock {\em Singular limits in thermodynamics of viscous fluids}.
\newblock Birkh{\" a}user--Basel, second edition, 2017.

\bibitem{FK2002}
J.~F\"urst and  K.~Kozel.
\newblock Numerical solution of transonic flows through 2D and 3D turbine cascades.
\newblock {\em Comput.~Visual.~Sci.} {\bf 4}: 183-189, 2002.


\bibitem{GGHL}
T.~Gallou{\"e}t, L.~Gastaldo, R.~Herbin, and J.~C.~Latch{\'e}.
\newblock An unconditionally stable pressure correction scheme for the compressible barotropic Navier--Stokes equations.
\newblock {\em ESAIM: M2AN }, {\bf 42}(2): 303--331, 2008.


\bibitem{GallouetMAC2018}
T.~Gallou{\"e}t, R. Herbin, J.-C. Latch\'e and D. Maltese.
\newblock Convergence of the MAC scheme for the compressible stationary Navier-Stokes equations.
\newblock {\em Math. Comp.} {\bf 87}(311): 1127--1163, 2018.

\bibitem{GallouetMAC}
T.~Gallou{\"e}t, D.~Maltese, and A.~Novotn{\'y}.
\newblock Error estimates for the implicit MAC scheme for the compressible Navier--Stokes equations.
\newblock{\em Numer. Math.} {\bf 141}(2): 495--567, 2019.

\bibitem{Grapsas}
D. Grapsas, R. Herbin,  W. Kheriji and J.-C. Latch\'e.
\newblock An unconditionally stable staggered pressure correction scheme for the compressible Navier--Stokes equations.
\newblock {\em SMAI J. Comput. Math.}, {\bf 2}: 51--97, 2016.


\bibitem{HJL}
J.~Haack, S.~Jin and J.~G.Liu.
\newblock An all--speed asymptotic--preserving method for the isentropic Euler and Navier-Stokes equations.
\newblock{\em Commun. Comput. Phys.} {\bf 12}(4): 955--980, 2012.


\bibitem{hoseksheMAC}
R.~Ho\v{s}ek and B.~She.
\newblock Stability and consistency of a finite difference scheme for compressible viscous isentropic flow in multi-dimension.
\newblock{\em J. Numer. Math.} {\bf 26}(3): 111--140, 2018.

\bibitem{ID2018}
M.~Ioriatti and M.~Dumbser.
\newblock Semi-implicit staggered discontinuous {G}alerkin schemes for
              axially symmetric viscous compressible flows in elastic tubes.
\newblock {\em Comput. \& Fluids} {\bf 167}, 166--179, 2018.

\bibitem{jovanovic}
V.~Jovanovi\'c.
\newblock An error estimate for a numerical scheme for the compressible Navier-Stokes system.
\newblock {\em Kragujevac J. Math.} {\bf 30}:263--275, 2007.

\bibitem{Karper}
T.~Karper.
\newblock A convergent FEM-DG method for the compressible Navier--Stokes equations.
\newblock {\em Numer. Math.} {\bf 125}(3): 441--510, 2013.

\bibitem{LPK2015}
P.~Louda, J.~P\v{r}\'ihoda, and K.~Kozel.
\newblock Numerical simulation of 3D backward facing step flows at various Reynolds
numbers.
\newblock {\em EPJ Web of Conferences} {\bf 92} 02049,  2015.


\bibitem{MS1998}
A.~Meister and T.~Sonar.
\newblock Finite-volume schemes for compressible flows.
\newblock {\em Surv.~Math.~Ind} {\bf 8}: 1--36, 1998.


%\bibitem{MiShe18}
%H.~Mizerov{\'a} and B.~She.
%\newblock Convergence of a finite difference scheme for the compressible viscous isentropic flow in multi-dimension. Preprint not ready.



\bibitem{NBALM}
S.~Noelle, G.~Bispen, K.~R.~Arun, M.~Luk\'a\v{c}ov\'a-Medvid'ov\'a, and C.-D.Munz,
\newblock A weakly asymptotic preserving low Mach number scheme for the Euler equations of gas dynamics.
\newblock{\em  SIAM J. Sci. Comput.} {\bf 36}(6): 989--1024, 2014.


\bibitem{PED1}
P.~Pedregal.
\newblock {\em Parametrized measures and variational principles}.
\newblock Birkh{\" a}user, Basel, 1997.

\bibitem{PKH2014}
P.~Po\v{r}\'{\i}zkov\'a, K.~Kozel, and J.~Hor\'a\v{c}ek.
\newblock Unsteady compressible flows in channel with varying
walls.
\newblock {\em J. Phys.: Conference Series} {\bf 490}:012066, 2014.



\bibitem{Sun}
Y.~Sun, C.~Wang, and Z.~Zhang.
\newblock A {B}eale-{K}ato-{M}ajda blow--up criterion for the 3--D compressible {N}avier--{S}tokes equations.
\newblock {\em J. Math. Pures. Appl.} {\bf 95}(1): 36--47, 2011.

\bibitem{W1997}
M.~Wierse.
\newblock A new theoretically motivated higher order upwind
scheme on unstructured grids of simplices.
\newblock {\em Adv. Comput. Math.} {\bf 7}: 303--335, 1997.

\bibitem{WK1996}
M.~Wierse and D.~Kr\"oner.
\newblock Higher order upwind schemes on unstructured grids for the
nonstationary compressible Navier-Stokes equations
in complex timedependent geometries in 3D.
\newblock {\em Flow simulation with high-performance computers, II,  Notes Numer. Fluid Mech.} {\bf 52}: 369--384, 1996.



\bibitem{ZTN2014}
O.~C.~Zienkiewicz,  R.~L.~Taylor, and P.~Nithiarasu.
\newblock {\em The Finite Element Method for Fluid Dynamics.}
\newblock Elsevier, 2014.

\end{thebibliography}
\bibliographystyle{siamplain}

\end{document}